\documentclass{amsart}
\usepackage{amsmath}
\usepackage{amsfonts}
\usepackage{amssymb}
\usepackage{latexsym}
\usepackage{graphicx}
\usepackage[final]{hyperref}
\hypersetup{colorlinks=true, linkcolor=blue, anchorcolor=blue, citecolor=red, filecolor=blue, menucolor=blue, pagecolor=blue, urlcolor=blue}

\newcommand{\roundb}[1]{\left(#1\right)} 

\newcommand{\Hproduct}[1]{%
  \roundb{\begin{smallmatrix}#1\end{smallmatrix}}
} %

\renewcommand{\d}{\, \mathrm{d}}

\newcommand{\init}{\vert_{t = 0}}

\newcommand{\Abs}[1]{\left\vert #1 \right\vert}
\newcommand{\abs}[1]{\vert #1 \vert}
\newcommand{\bigabs}[1]{\bigl\vert #1 \bigr\vert}

\newcommand{\norm}[1]{\left\Vert #1 \right\Vert}

\newcommand{\bignorm}[1]{\bigl\Vert #1 \bigr\Vert}

\newcommand{\C}{\mathbb{C}}

\newcommand{\R}{\mathbb{R}}

\newcommand{\angles}[1]{\langle #1 \rangle}

\DeclareMathOperator{\diag}{diag}

\DeclareMathOperator{\im}{Im}
\DeclareMathOperator{\re}{Re}

\newtheorem{theorem}{Theorem}[section]

\newtheorem{lemma}[theorem]{Lemma}

\theoremstyle{definition}
\newtheorem{definition}[theorem]{Definition}

\theoremstyle{remark}
\newtheorem{remark}[theorem]{Remark}

\numberwithin{equation}{section}
\title[Finite energy M-K-G in Lorenz gauge]{Finite-energy global well-posedness of the Maxwell-Klein-Gordon system in Lorenz gauge}

\author{Sigmund Selberg}
\address{Department of Mathematical Sciences\\ Norwegian University of Science and Technology\\ N-7491 Trondheim\\ Norway}
\email{sselberg@math.ntnu.no}
\urladdr{www.math.ntnu.no/~sselberg}

\thanks{The first author acknowledges partial support by the Research Council of Norway, grants 160192 and 185359. Second author supported by the Research Council of Norway, grant 160192.}

\author{Achenef Tesfahun}
\address{Department of Mathematical Sciences\\ Norwegian University of Science and Technology\\ N-7491 Trondheim\\ Norway}

\begin{document}

\begin{abstract} 
It is known that the Maxwell-Klein-Gordon system (M-K-G), when written relative to the Coulomb gauge, is globally well-posed for finite-energy initial data. This result, due to Klainerman and Machedon, relies crucially on the null structure of the main bilinear terms of M-K-G in Coulomb gauge. It appears to have been believed that such a structure is not present in Lorenz gauge, but we prove here that it is, and we use this fact to prove finite-energy global well-posedness in Lorenz gauge. The latter has the advantage, compared to Coulomb gauge, of being Lorentz invariant, hence M-K-G in Lorenz gauge is a system of nonlinear wave equations, whereas in Coulomb gauge the system has a less symmetric form, as it contains also a nonlinear elliptic equation.
\end{abstract}

\maketitle


\section{Introduction}\label{A}

Points in Minkowski space $\R^{1+3}$ are written $(x^0,x^1,x^2,x^3)$, and $\partial_\mu$ denotes the partial derivative with respect to $x^\mu$. Often we split the coordinates into the time variable $t=x^0$ and the space variable $x=(x^1,x^2,x^3)$, and we write $\partial_t=\partial_0$ and $\nabla = (\partial_1,\partial_2,\partial_3)$. Roman indices $j,k,\dots$ run over $1,2,3$, greek indices $\mu,\nu,\dots$ over $0,1,2,3$, and repeated upper/lower indices are implicitly summed over these ranges. Indices are raised and lowered using the Minkowski metric $\diag(-1,1,1,1)$. 

The Maxwell-Klein-Gordon system (M-K-G) describes the motion of a spin-0 particle self-interacting with an electromagnetic field. It is obtained by coupling Maxwell's equation for the electric and magnetic fields $\mathbf E, \mathbf B \colon \R^{1+3} \to \R^3$ with the Klein-Gordon equation for a scalar field $\phi \colon \R^{1+3} \to \C$, and reads
\begin{gather}
  \label{A:2}
  \nabla \cdot \mathbf E = \rho,
  \qquad
  \nabla \cdot \mathbf B = 0,
  \qquad
  \nabla \times \mathbf E + \partial_t \mathbf B = 0,
  \qquad
  \nabla \times \mathbf B - \partial_t \mathbf E = \mathbf J,
  \\
  \label{A:4}
  D^{(A)}_\mu D^{(A)\mu} \phi = m^2\phi,
\end{gather}
where $m > 0$ is a constant and
\begin{equation}\label{A:5}
  D_\mu^{(A)} = \partial_\mu - iA_\mu
\end{equation}
is the gauge covariant derivative corresponding to a real-valued 4-potential $A = \{A_\mu \}_{\mu = 0,1,2,3}$ representing the electromagnetic field:
\begin{equation}\label{A:6}
  \mathbf B = \nabla \times \mathbf A,
  \qquad
  \mathbf E = \nabla A_0 - \partial_t \mathbf A,
\end{equation}
where $\mathbf A = (A_1,A_2,A_3)$ is the spatial part of $A$. To complete the coupling, we specify the charge density $\rho$ and the current density $\mathbf J$ in \eqref{A:2}, namely
\begin{equation}\label{A:8}
\begin{aligned}
  \rho
  &=
  - \im \left( \phi \overline{D_0^{(A)} \phi} \, \right)
  = - \im \left( \phi \overline{\partial_t \phi} \, \right) - \abs{\phi}^2 A_0,
  \\
  \mathbf J
  &=
  \im \left( \phi \overline{\nabla^{(A)} \phi} \, \right)
  =
  \im \left( \phi \overline{\nabla \phi} \, \right) + \abs{\phi}^2 \mathbf A,
\end{aligned}
\end{equation}
where $\nabla^{(A)} = (D_1^{(A)},D_2^{(A)},D_3^{(A)})$. This choice of densities is dictated by the natural 4-current density associated to the Klein-Gordon equation:
$$
  J_\mu = \im \left( \phi \overline{D_\mu^{(A)} \phi} \, \right).
$$
Note that $\rho = J^0 = - J_0$ and $\mathbf J = (J^1,J^2,J^3) = (J_1,J_2,J_3)$.

Equations \eqref{A:2} and \eqref{A:4}, coupled by \eqref{A:5}--\eqref{A:8}, constitute the M-K-G system.

The total energy of a solution, at time $t$, is
$$
  \mathcal E(t)
  = \frac12
  \int_{\R^3}
  \left(
  \bigabs{D^{(A)} \phi(t,x)}^2 + m^2\abs{\phi(t,x)}^2
  + \Abs{\mathbf E(t,x)}^2
  + \Abs{\mathbf B(t,x)}^2 \right) \d x.
$$
For a smooth solution decaying sufficiently fast at spatial infinity, the energy is a conserved quantity.

Formally, the second and third equations in \eqref{A:2} are equivalent to the existence of a potential $A$ satisfying \eqref{A:6}, but this potential is not unique. In fact, M-K-G is invariant under the \emph{gauge transformation}
\begin{equation}\label{A:90}
  \phi \longrightarrow \phi' = e^{i\chi} \phi,
  \qquad A_\mu \longrightarrow A_\mu' = A_\mu + \partial_\mu \chi,
\end{equation}
for any sufficiently smooth $\chi : \R^{1+3} \to \R$. That is, if $(\phi,A)$ satisfies M-K-G, then so does $(\phi',A')$, as can be seen from the identity $D_\mu^{(A')}\phi' = e^{i\chi}D_\mu^{(A)}\phi$.

Since the observables $\mathbf E$, $\mathbf B$, $\rho$ and $\mathbf J$ are not affected by a gauge transformation, two solutions related by such a transformation are physically undistinguishable, and must be considered equivalent. In practice, a solution is therefore a representative of its equivalence class, hence we have \emph{gauge freedom}: We are free to choose a representative that suits our needs.

Note that if we express also the first and fourth equations in \eqref{A:2} in terms of the potential $A$, then \eqref{A:2} is replaced by
\begin{equation}\label{A:100}
  \square A_\mu - \partial_\mu ( \partial^\nu A_\nu ) = - J_\mu
  \qquad \left( \square = \partial_\mu \partial^\mu = -\partial_t^2 + \Delta \right).
\end{equation}
In view of the gauge freedom, we can impose an additional \emph{gauge condition} on $A$, which simplifies the analysis as much as possible. Looking at \eqref{A:100}, an obvious choice is the \emph{Lorenz gauge condition}, $\partial^\mu A_\mu = 0$, or equivalently, $\partial_t A_0 = \nabla \cdot \mathbf A$. Note that this does not uniquely determine $A$: The condition $\partial^\mu A_\mu = 0$ is preserved by the gauge transformation \eqref{A:90} for any $\chi$ satisfying $\square \chi = 0$.

In Lorenz gauge, the M-K-G system becomes
\begin{equation}\label{A:104}
\left\{
\begin{aligned}
  &D^{(A)}_\mu D^{(A)\mu} \phi = m^2\phi,
  \\
  &\square A = - \im \left( \phi \overline{D^{(A)} \phi} \, \right) ,
  \\
  &\partial^\mu A_\mu = 0.
\end{aligned}
\right.
\end{equation}

Another popular choice of gauge is the \emph{Coulomb gauge condition}, $\nabla \cdot \mathbf A = 0$, or equivalently $\mathbf P \mathbf A = \mathbf A$, where $\mathbf P$ denotes the projection onto the divergence free vector fields on $\R^3$. Then \eqref{A:100} splits into a nonlinear elliptic equation $\Delta A_0 = \rho$ and the nonlinear wave equation $\square \mathbf A = - \mathbf P \mathbf J$. The Coulomb gauge was used by Klainerman and Machedon in \cite{Klainerman:1994b}, where the global well-posedness of M-K-G for finite-energy data was proved (thus they recovered, in particular, the earlier global regularity result from \cite{Eardley:1982}). The key observation made in \cite{Klainerman:1994b} was that in Coulomb gauge the main bilinear terms in M-K-G have a so-called null structure, which cancels the worst interactions in a product of two waves. Without this structure, there would be no hope of proving local well-posedness of M-K-G in the energy class. Since the seminal work \cite{Klainerman:1994b}, it appears to have been widely believed that Coulomb gauge is distinguished with respect to the presence of null structure. But as we show in this paper, the structure is there also in Lorenz gauge. The first instance of null structure in Lorenz gauge for a nonlinear field theory was found for the Maxwell-Dirac system by D'Ancona, Foschi and the first author in \cite{Selberg:2008b}, which inspired the present work.

The main advantage of Lorenz gauge is its Lorentz invariance, resulting in a more symmetric form of M-K-G than in Coulomb gauge, where one has to deal with the nonlinear elliptic equation for $A_0$, and the nonlocality of the system.

On the other hand, one may argue that Coulomb gauge has the advantage that the potential $\mathbf A$ is uniquely determined by $\mathbf B$, and $\mathbf A$ gains one degree of Sobolev regularity compared to $\mathbf B$, whereas in Lorenz gauge the potential is not uniquely determined and appears to lose regularity compared to its initial data. This is not a problem, however, since the regularity of $A$ as such is not an issue: Only the observables $\mathbf E, \mathbf B$ represented by $A$ are of interest, and these do not lose uniqueness or regularity, as we show here.

This paper is organized as follows: In the next section we pose the correct initial value problem for M-K-G in Lorenz gauge, with finite-energy data, and we state our main theorem, which is that this problem is globally well-posed. This is then the analogue, in Lorenz gauge, of the result obtained in Coulomb gauge in \cite{Klainerman:1994b}. In section \ref{N} we demonstrate the null structure in Lorenz gauge, and compare it to the structure found in Coulomb gauge in \cite{Klainerman:1994b}. In sections \ref{L} and \ref{LL} we use this structure to prove local existence in the energy class, the main technical tool being product estimates in wave-Sobolev spaces. Using the conservation of energy we show in section \ref{G} that the local result extends to a global one, and finally in section \ref{U} we prove uniqueness of the solution.

Some notation: $H^s$ (for any $s \in \R$) and $\dot H^s$ (for $\abs{s} < \frac32$) are the completions of the Schwartz space $\mathcal S(\R^3)$ with respect to the norms $\norm{f}_{H^s} = \bignorm{\angles{\xi}^s \widehat f\,}_{L^2}$ and $\norm{f}_{\dot H^s} = \bignorm{\abs{\xi}^s \widehat f\,}_{L^2}$, respectively, where $\widehat f(\xi) = \mathcal F f(\xi)$ is the Fourier transform of $f(x)$ and we use the shorthand $\angles{\xi} = (1+\abs{\xi}^2)^{\frac12}$. We shall make frequent use of the embedding $\dot H^1 \hookrightarrow L^6$, where $\hookrightarrow$ denotes continuous inclusion. The space $\abs{\nabla}^{-1} H^s$ is defined by
$$
  \abs{\nabla}^{-1} H^s
  = \mathcal F^{-1} \left\{ \frac{\widehat g(\xi)}{\abs{\xi}\angles{\xi}^s} \colon g \in L^2(\R^3) \right\}
  = \mathcal F^{-1} L^2 \left( \abs{\xi}^2\angles{\xi}^{2s} \, d\xi \right)
$$
with norm $\bignorm{\abs{\xi}\angles{\xi}^s\widehat f\,}_{L^2}$. Equivalently, $\abs{\nabla}^{-1} H^s$ is the completion of $\mathcal S(\R^3)$ with respect to this norm. Note that $\dot H^1 = \abs{\nabla}^{-1} H^0 = \abs{\nabla}^{-1} L^2$. By splitting into low and high frequencies ($\abs{\xi} \le 1$ and $\abs{\xi} > 1$), and using the embedding $\dot H^1 \hookrightarrow L^6$, we get $\abs{\nabla}^{-1} H^s \hookrightarrow L^6 + H^{s+1}$.

In estimates we use the shorthand $X \lesssim Y$ for $X \le CY$, where $C \gg 1$ is a constant which may depend on quantities which are considered fixed, such as the exponents of Sobolev norms involved in the estimate. Further, $X=O(R)$ is short for $\abs{X} \lesssim R$, $X \sim Y$ means $X \lesssim Y \lesssim X$, and $X \ll Y$ stands for $X \le C^{-1} Y$, with $C$ as above. We write $\simeq$ for equality up to multiplication by an absolute constant (typically factors involving $2\pi$, in connection with the Fourier transform).

\section{The main result}\label{M}

We are interested in the Cauchy problem starting from data
\begin{equation}\label{M:100}
  \phi\init = \phi_0,
  \qquad
  D^{(A\init)}\phi\init = U,
  \qquad
  \mathbf E\init = \mathbf E_0,
  \qquad
  \mathbf B\init = \mathbf B_0,
\end{equation}
such that the initial energy,
\begin{equation}\label{M:120}
  \mathcal E(0) =
  \frac12
  \int_{\R^3}
  \left(
  \abs{U(x)}^2 +  m^2\abs{\phi_0(x)}^2
  + \Abs{\mathbf E_0(x)}^2
  + \Abs{\mathbf B_0(x)}^2 \right) \, dx,
\end{equation}
is finite. In view of \eqref{A:2} and \eqref{A:8}, we must assume
\begin{equation}\label{M:130}
  \nabla \cdot \mathbf E_0
  = \rho_0 \equiv - \im \left( \phi_0 \overline{ U_0 } \right),
  \qquad
  \nabla \cdot \mathbf B_0 =0,
\end{equation}
where $U_0$ is the first component of $U = (U_0,U_1,U_2,U_3) = (U_0,\mathbf U)$.
 
Even with the Lorenz condition imposed, there is still some gauge freedom: The initial covariant derivative in \eqref{M:100} depends on $A \init$, which is not uniquely determined. There is, however, a natural choice of $A\init$ which allows one to control it by the energy, as we now discuss.

Writing $(A,\partial_t A)\init =(a,\dot a)$, and denoting the spatial part by $(\mathbf a,\dot{\mathbf a})$, we get from the Lorenz condition and Maxwell's equations the constraints
\begin{equation}\label{M:140}
  \dot a_0 = \nabla \cdot \mathbf a.
  \qquad
  \mathbf B_0 = \nabla \times \mathbf a,
  \qquad
  \mathbf E_0 = \nabla a_0 - \dot{\mathbf a}.
\end{equation}
Set $a_0 = \dot a_0 = 0$. Then \eqref{M:140} uniquely determines $\mathbf a \in \dot H^1$ and $\dot{\mathbf a} \in L^2$, and
\begin{equation}\label{M:160}
  \norm{\mathbf a}_{\dot H^1} \le C \norm{\mathbf B_0}_{L^2},
  \qquad
  \norm{\dot{\mathbf a}}_{L^2} = \norm{\mathbf E_0}_{L^2},
\end{equation}
assuming $\mathcal E(0) < \infty$.

\begin{remark}\label{A:Rem1} Our choice of initial gauge is justified by gauge freedom: Suppose $(\phi,A)$ is a solution of \eqref{A:104}. Let $\chi$ be the solution of
\begin{equation}\label{M:170}
  \square \chi = 0,
  \qquad
  \Delta \chi(0) = - \nabla \cdot \mathbf a,
  \qquad
  \partial_t \chi(0) = - a_0,
\end{equation}
and apply the gauge transformation \eqref{A:90}. Since $\square \chi = 0$, the Lorenz condition is preserved, and by the choice of data for $\chi$ we have $a_0' = 0$ and $\nabla \cdot \mathbf a' = 0$; then in view of the Lorenz condition, we further have $\dot{a}'_0 = 0$.
\end{remark}

\begin{remark}\label{A:Rem2} 
Although $\nabla \cdot \mathbf a = 0$, our initial gauge is \emph{not} Coulomb. Indeed, the latter requires a special choice of $a_0$ obtained by solving a certain elliptic equation, and is not compatible with our choice $a_0=0$.
\end{remark}

By our choice of initial gauge, we get the constraint
\begin{equation}\label{M:180}
  U = (U_0,\mathbf U) = \left( \phi_1, \nabla \phi_0 - i \phi_0 \mathbf a \right),
\end{equation}
where we write $\phi_1 = \partial_t \phi\init$. Then the assumption $\mathcal E(0) < \infty$ implies $\phi_0 \in H^1$, as the following lemma shows.

\begin{lemma}\label{M:Lemma} Suppose we are given $\phi_0 \in L^2(\R^3;\C)$ and $\mathbf a \in \dot H^1(\R^3;\R^3)$. Define
$$
  \mathbf U = \nabla \phi_0 - i \phi_0 \mathbf a,
$$
and assume that $\mathbf U \in L^2$. Then $\phi_0 \in H^1$, and
\begin{equation}\label{M:190}
  \norm{\nabla \phi_0}_{L^2} \le 2 \norm{\mathbf U}_{L^2}
  + C \norm{\mathbf a}_{\dot H^1}^2\norm{\phi_0}_{L^2},
\end{equation}
where $C$ is an absolute constant.
\end{lemma}

\begin{proof} Write $\nabla \phi_0 = \mathbf U + \mathbf V$, where $\mathbf V = i \phi_0 \mathbf a$. Since $\mathbf a \in \dot H^1 \subset L^6$, we have $\mathbf V \in L^{\frac32}$, hence $\nabla \phi_0 \in L^2 + L^{\frac32}$. We claim that this implies $\phi_0 \in L^6 + L^3$. Granting this for the moment, we get $\phi_0 \in (L^6 + L^3) \cap L^2 \subset L^3$, but this implies $\mathbf V \in L^2$, hence $\nabla \phi_0 \in L^2$, which is what we wanted.

To prove the claim, $\phi_0 \in L^6 + L^3$, we use the fact (see \cite[Ch.~V]{Stein:1970}) that
\begin{equation}\label{M:200}
  \abs{f} \le C I_1(\abs{\nabla f})
  \qquad \text{for all $f \in C_c^\infty(\R^3)$},
\end{equation}
where $I_1 f = (-\Delta)^{-\frac12} f = \frac{\gamma}{\abs{x}^2} * f$ is a Riesz potential (here $\gamma > 0$ is a constant). By the $L^p$ inequality for potentials (see \cite[Ch.~V, Thm.~1]{Stein:1970}), $\norm{I_1(\abs{\mathbf U})}_{L^6} \le C \norm{\mathbf U}_{L^2}$ and $\norm{I_1(\abs{\mathbf V})}_{L^3} \le C \norm{\mathbf V}_{L^{\frac32}}$, so via a regularization we get from \eqref{M:200} that
$$
  \abs{\phi_0} \le C I_1(\abs{\nabla \phi_0})
  \le CI_1(\abs{\mathbf U}) + CI_1(\abs{\mathbf V})
$$
a.e.~in $\R^3$. Thus, $\abs{\phi_0} \le f + g$, where $f \in L^6$ and $g \in L^3$, hence $\phi_0 \in L^6 + L^3$.

Finally, we prove \eqref{M:190}:
\begin{align*}
  \norm{\nabla \phi_0}_{L^2}
  &\le \norm{\mathbf U}_{L^2} + \norm{\phi_0}_{L^3}\norm{\mathbf a}_{L^6}
  \\
  &\le \norm{\mathbf U}_{L^2} +
  \norm{\phi_0}_{L^6}^{1/2}\norm{\phi_0}_{L^2}^{1/2}\norm{\mathbf a}_{L^6}
  \\
  &\le \norm{\mathbf U}_{L^2} +
  C\norm{\nabla\phi_0}_{L^2}^{1/2}\norm{\phi_0}_{L^2}^{1/2}\norm{\mathbf a}_{\dot H^1}
  \\
  &\le \norm{\mathbf U}_{L^2} +
  \frac12\norm{\nabla\phi_0}_{L^2}
  +
  C'\norm{\phi_0}_{L^2}\norm{\mathbf a}_{\dot H^1}^2,
\end{align*}
where Young's inequality was used at the end.
\end{proof}

In view of Lemma \ref{M:Lemma}, and our choice of initial gauge, which guarantees $\mathbf a \in \dot H^1$, the assumption $\mathcal E(0) < \infty$ is equivalent to
\begin{equation}\label{M:210}
\left\{
\begin{aligned}
  \phi\init &= \phi_0 \in H^1(\R^3;\C),
  \\
  \partial_t\phi\init &= \phi_1 \in L^2(\R^3;\C),
  \\
  \mathbf E\init &= \mathbf E_0 \in L^2(\R^3;\R^3),
  \\
  \mathbf B\init &= \mathbf B_0 \in L^2(\R^3;\R^3).
\end{aligned}
\right.
\end{equation}
Note also that, since $U_0=\phi_1$ (recall \eqref{M:180}), the constraint \eqref{M:130} becomes
\begin{equation}\label{M:220}
  \nabla \cdot \mathbf E_0
  = - \im \left( \phi_0 \overline{ \phi_1 } \right),
  \qquad
  \nabla \cdot \mathbf B_0 = 0.
\end{equation}

\begin{remark}\label{A:Rem3}
The role of the constraint \eqref{M:220} is to fix the curl-free parts of $\mathbf E_0$ and $\mathbf B_0$ in $L^2$. Then $\phi_0 \in H^1$ and $\phi_1 \in L^2$ can be chosen arbitrarily, as can the divergence-free parts of $\mathbf E_0$ and $\mathbf B_0$ in $L^2$. Indeed, \eqref{M:220} determines the curl-free part of $\mathbf E_0$, namely $\mathbf E_0^{\text{cf}}
  =
  \Delta^{-1} \nabla ( \nabla \cdot \mathbf E_0 )
  =
  - \Delta^{-1} \nabla \im \left( \phi_0 \overline{ \phi_1 } \right),
$
and this belongs to $L^2$ since $(-\Delta)^{-\frac12} \nabla$ is bounded on $L^2$ and
$$
  \norm{(-\Delta)^{-\frac12} \left( \phi_0 \overline{ \phi_1 } \right)}_{L^2}
  \le C
  \norm{\phi_0 \overline{ \phi_1 }}_{L^{\frac65}}
  \le C
  \norm{\phi_0}_{L^3}
  \norm{\phi_1}_{L^2}
  \le C'
  \norm{\phi_0}_{H^1}
  \norm{\phi_1}_{L^2},
$$
where we used again the $L^p$ inequality for potentials (see \cite[Ch.~V, Thm.~1]{Stein:1970}).
\end{remark}

We can now state our main result.

\begin{theorem}\label{M:Thm}
Given finite energy data \eqref{M:210} satisfying \eqref{M:220}, set $a_0=\dot a_0 = 0$, and let $(\mathbf a ,\dot{\mathbf a}) \in \dot H^1(\R^3;\R^3) \times L^2(\R^3;\R^3)$ be the unique solution of \eqref{M:140}.

There exists a unique global solution
\begin{align*}
  \phi &\in C\bigl(\R;H^1(\R^3;\C)\bigr) \cap C^1\bigl(\R;L^2(\R^3;\C)\bigr),
  \\
  \mathbf E, \mathbf B &\in C\bigl(\R;L^2(\R^3,\R^3)\bigr),
\end{align*}
of the M-K-G system \eqref{A:2}--\eqref{A:8} with initial condition \eqref{M:210}, relative to a real-valued 4-potential $A$ such that
$$
  A,\partial_t A \in C\bigl(\R;\mathcal D'(\R^3)\bigr),
  \qquad
  \partial^\mu A_\mu = 0,
  \qquad
  (A,\partial_t A)\init=(a,\dot a),
$$
and such that the total energy is conserved:
$$
  \mathcal E(t) = \mathcal E(0) \qquad \text{for all $t$},
$$
where $\mathcal E(0)$ is given by \eqref{M:120}, with $U$ defined by \eqref{M:180}.
\end{theorem}

For the potential $A$ we can only prove the regularity
$A \in C(\R;\dot H^1 + H^{1-\delta})$ and $\partial_t A \in C(\R;H^{-\delta})$ for any $\delta > 0$, suggesting a small loss of regularity compared to the data $(a,\dot a)$. But as noted, the regularity of $A$ as such is not of interest.

The main step in the proof of Theorem \ref{M:Thm} is to obtain local well-posedness; the global result then follows by conservation of energy. By a contraction argument, the local well-posedness is reduced to proving certain nonlinear estimates in $X^{s,b}$ spaces adapted to the linear part of the evolution, and here the null structure in Lorenz gauge is crucial: without it, the estimates would be just out of reach. Once we have the required structure, which is the main new contribution made here, we can reduce to known product estimates for the $X^{s,b}$ spaces in question.

The contraction argument does not give the unconditional uniqueness as stated in Theorem \ref{M:Thm}, but only uniqueness in the smaller contraction space. To prove the unconditional uniqueness, we use arguments similar to those in \cite{Zhou:2000} and \cite{Masmoudi:2003b}, which rely fundamentally on the fact that the problem we are considering is subcritical; for \mbox{M-K-G}, the scale invariant regularity for $\phi$ is $\dot H^{\frac12}$, whereas the energy corresponds to $H^1$, so the problem is energy-subcritical. In this connection, we mention that in Coulomb gauge, local well-posedness has been proved almost all the way down to the critical regularity; see \cite{Machedon:2004}, and also \cite{Cuccagna:1999}. We do not investigate here how far down in regularity one can go in Lorenz gauge, but it is clear from our proof that one can go at least some way below energy (there is some headroom in all the estimates). Low regularity results have also been obtained (for the Yang-Mills equations) in the temporal gauge, but are limited to small data; see \cite{Tao:2003}.

We remark that, since local well-posedness is proved by a contraction argument based on estimates in $X^{s,b}$ type spaces, it is a standard fact that the solutions in Theorem \ref{M:Thm} enjoy continuous dependence on the data and persistence of higher regularity: $(\phi,\partial_t \phi,\mathbf E,\mathbf B)$ depends continuously on the data \eqref{M:210}, \eqref{M:220} in $X = H^1 \times L^2 \times L^2 \times L^2$, locally uniformly in time with values in the same space $X$, and higher Sobolev regularity of the data persists for all time, in the sense that if the data belong to $X^k = H^{1+k} \times H^k \times H^k \times H^k$ for some $k > 0$, then the solution describes a continuous curve in $X^k$ for all time. In particular, this means that our solutions are limits, again locally uniformly in time with values in $X$, of smooth solutions with $C_c^\infty$ initial data.

\section{Null structure of M-K-G in Lorenz gauge}\label{N}

Using the definition \eqref{A:5} of $D^{(A)}_\mu$, and the Lorenz gauge condition $\partial^\mu A_\mu = 0$, we write the first two equations in \eqref{A:104} as
\begin{equation}\label{N:100}
\left\{
\begin{aligned}
  (\square - m^2)\phi &=
  \mathcal M(A,\phi)
  \equiv
  2iA^\mu\partial_\mu\phi + A_\mu A^\mu \phi,
  \\
  \square A &= \mathcal N(A,\phi)
  \equiv
  - \im \left( \phi \overline{\partial \phi} \, \right)
  - A \abs{\phi}^2.
\end{aligned}
\right.
\end{equation}
The key terms here are the bilinear ones; the cubic terms turn out to be much easier to control, since they do not contain derivatives.

\subsection{Null structure of the term $A^\mu\partial_\mu\phi$} Recall the splitting of $\mathbf A$ (or indeed any
vector field on $\R^3$) into its divergence-free and curl-free parts:
\begin{equation}\label{N:110}
\mathbf A = \underbrace{-\Delta^{-1} \nabla \times \nabla
\times \mathbf A}_{\text{divergence-free}}
  + \underbrace{\Delta^{-1} \nabla ( \nabla \cdot \mathbf A )}_{\text{curl-free}}
  \equiv
  \mathbf A^{\text{df}} + \mathbf A^{\text{cf}}.
\end{equation}
Now expand:
\begin{equation}\label{N:130}
  A^\mu\partial_\mu\phi
  = \left( - A_0 \partial_t \phi + \mathbf A^{\text{cf}} \cdot \nabla \phi \right)
  + \mathbf A^{\text{df}} \cdot \nabla \phi
  \equiv P_1 + P_2.
\end{equation}

The term $P_2$ was shown by Klainerman and Machedon \cite{Klainerman:1994b} to be a null form:
\begin{multline}\label{N:140}
  P_2 = \mathbf A^{\text{df}} \cdot \nabla \phi
  =
  - \epsilon^{jkl} \partial_k w_l  \partial_j
\phi
  =
  (\nabla w_l\times \nabla \phi)^l,
  \\ \text{where $\mathbf w = \Delta^{-1} \nabla \times \mathbf A = \Delta^{-1} \mathbf B$}.
\end{multline}
Here $w_l = w^l$ denotes the $l$-th component of a 3-vector $\mathbf w$, $\epsilon^{jkl}$ is the Levi-Civita symbol, and we use the summation convention.

Our main new observation is that the term $P_1$ is also a null form. Note that this term was not an issue in \cite{Klainerman:1994b}, since there the Coulomb gauge was used, and then $\mathbf A^{\text{cf}} = 0$, whereas $A_0$ solves an elliptic equation and therefore has sufficiently good regularity properties so that no special structure is needed to control $A_0\partial_t \phi$. In Lorenz gauge, on the other hand, the term $P_1$ must be dealt with, and we now show how this is done.

First observe that if we assume the Lorenz gauge condition, then $\partial_t A_0 = \nabla \cdot \mathbf A$, hence $\mathbf A^{\text{cf}} = \Delta^{-1} \nabla \partial_t A_0$. Therefore,
\begin{equation}\label{N:150}
  P_1 = - A_0 \partial_t \phi + \mathbf A^{\text{cf}} \cdot \nabla \phi
  = - A_0 \partial_t \phi 
  +
  \Delta^{-1} \nabla \partial_t A_0 \cdot \nabla \phi.
\end{equation}
Now if we denote the space-time Fourier variables of $A_0$ and $\phi$ by $(\tau,\xi)$ and $(\lambda,\eta)$, respectively, where $\tau,\lambda \in \R$ are the temporal frequencies and $\xi,\eta \in \R^3$ are the spatial frequencies, then the symbol of $P_1$ is $-i\lambda + i \tau \frac{\xi\cdot\eta}{\abs{\xi}^2}$, and this vanishes if $(\tau,\xi)$ and $(\lambda,\eta)$ are parallel null vectors, that is, if $\tau=\pm\abs{\xi}$ and $(\lambda,\eta) = c (\tau,\xi)$ for some $c \in \R$. Thus, $P_1$ is a null form. It is possible to quantify the null cancellation in terms of the angle between the space-time frequencies, but it turns out to be more convenient to get rid of the temporal frequencies in the null form symbol, by splitting $A_0$ and $\phi$ using standard spectral decompositions for the wave and Klein-Gordon equations, which we pause to recall here.

Let us start with the homogeneous wave equation $\square u = 0$, which can be written as a first order system:
$
  \frac{\partial}{\partial t} \left( \begin{smallmatrix}
          u \\
        u_t \\
     \end{smallmatrix} \right)
  =
  \left( \begin{smallmatrix}
  0 & 1 \\
  \Delta & 0
  \end{smallmatrix} \right)
  \left( \begin{smallmatrix}
        u \\
        u_t \\
  \end{smallmatrix} \right),
$
where $u_t = \partial_t u$. Diagonalizing the symbol
$
  \left(\begin{smallmatrix}
  0 & 1 \\
  -\abs{\xi}^2 & 0
  \end{smallmatrix}\right),
$
one is led to the transformation $(u,u_t) \to (u_+,u_-)$, where
$$
  u_\pm = \frac12 \left( u \pm (i\abs{\nabla})^{-1} u_t \right).
$$
Here $\abs{\nabla} = \sqrt{-\Delta}$ is the multiplier with symbol $\abs{\xi}$. Obviously,
$$
  u= u_+ + u_-, \qquad u_t = i\abs{\nabla}( u_+ - u_- ),
$$
and $\square u = 0$ splits into the pair of equations $\left(i\partial_t \pm \abs{\nabla}\right) u_\pm = 0$. More generally, the inhomogeneous equation $\square u = F$ splits into the pair of equations
$$
  \left(i\partial_t \pm \abs{\nabla}\right) u_\pm = - (\pm 2\abs{\nabla})^{-1} F.
$$
Moreover, the same statements hold for the Klein-Gordon equation $(\square - m^2) u = F$, except that $\abs{\nabla}$ is then replaced by the multiplier $\angles{\nabla}_m = \sqrt{m^2 - \Delta}$ with symbol
$
  \angles{\xi}_m
  = \sqrt{m^2 + \abs{\xi}^2}.
$

Applying these decompositions to \eqref{N:100}, we therefore define
\begin{equation}\label{N:200}
\left\{
\begin{aligned}
  \phi_\pm &= \frac12 \left( \phi \pm (i\angles{\nabla}_m)^{-1} \phi_t \right),
  \\
  A_\pm &= \frac12 \left( A \pm (i\abs{\nabla})^{-1} A_t \right),
\end{aligned}
\right.
\end{equation}
thereby transforming \eqref{N:100} to
\begin{equation}\label{N:205}
\left\{
\begin{aligned}
  \left(i\partial_t \pm \angles{\nabla}_m\right) \phi_\pm
  &=
  - (\pm 2\angles{\nabla}_m)^{-1} \mathcal M(A,\phi),
  \\
  \left(i\partial_t \pm \abs{\nabla}\right) A_\pm
  &=
  - (\pm 2\abs{\nabla})^{-1} \mathcal N(A,\phi).
\end{aligned}
\right.
\end{equation}
Now $\phi = \phi_+ + \phi_-$, $\partial_t \phi = i\angles{\nabla}_m (\phi_+ - \phi_-)$, $A_0 = A_{0,+} + A_{0,-}$ and $\partial_t A_0 = i\abs{\nabla} (A_{0,+} - A_{0,-})$, so from \eqref{N:150} we get
\begin{equation}\label{N:210}
\begin{aligned}
  iP_1
  &= (A_{0,+} + A_{0,-}) \angles{\nabla}_m (\phi_+ - \phi_-)
  +
  \abs{\nabla}^{-1} \nabla (A_{0,+} - A_{0,-}) \cdot \nabla (\phi_+ + \phi_-)
  \\
  &=
  \sum_{\pm_1,\pm_2}
  \pm_2 \mathfrak A_{(\pm_1,\pm_2)}(A_{0,\pm_1}, \phi_{\pm_2}),
\end{aligned}
\end{equation}
where
\begin{equation}\label{N:220}
  \mathfrak{A_{(\pm_1,\pm_2)}}(f,g)
  =
  f \angles{\nabla}_m g
  +
  \abs{\nabla}^{-1} \nabla (\pm_1 f) \cdot
  \nabla (\pm_2 g)
\end{equation}
is a bilinear operator acting on functions of $x$. Fourier transformation in $x$ gives
\begin{equation}\label{N:230}
  \mathcal F\left( \mathfrak{A_{(\pm_1,\pm_2)}}(f,g) \right)(\xi)
  =
  \int_{\R^3} \mathfrak a_{(\pm_1,\pm_2)}(\eta,\xi-\eta)
  \widehat f(\eta) \widehat g(\xi-\eta) \d \eta,
\end{equation}
where
\begin{equation}\label{N:240}
  \mathfrak a_{(\pm_1,\pm_2)}(\eta,\zeta)
  \simeq 
  \angles{\zeta}_m
  -\frac{(\pm_1\eta) \cdot (\pm_2\zeta)}{\abs{\eta}}.
\end{equation}
The following estimate shows that, up to a harmless lower order term due to the positive mass $m$ in the Klein-Gordon equation, $\mathfrak{A_{(\pm_1,\pm_2)}}$ is a null form in the sense of \cite{Selberg:2007d}. We write $\theta(\eta,\zeta)$ for the angle between nonzero vectors $\eta,\zeta \in \R^3$.

\begin{lemma}\label{N:Lemma1}
$\Abs{\mathfrak a_{(\pm_1,\pm_2)}(\eta,\zeta)} \lesssim m + \abs{\zeta}\theta(\pm_1\eta,\pm_2\zeta)$ for all nonzero $\eta,\zeta \in \R^3$.
\end{lemma}

\begin{proof}
Without loss of generality take $\pm_1=\pm_2=+$. Then we estimate
$$
  \angles{\zeta}_m
  - \frac{\eta \cdot \zeta}{\abs{\eta}}
  =
  \left(\sqrt{m^2+\abs{\zeta}^2} - \abs{\zeta}\right)
  + \abs{\zeta} \left( 1 - \frac{\eta\cdot\zeta}{\abs{\eta}\abs{\zeta}} \right)
  = O(m) + \abs{\zeta}\, O\left(\theta(\eta,\zeta)^2\right).
$$
Note that this is actually stronger than what we need, since the angle is squared.
\end{proof}

\subsection{Null structure in the Maxwell part}\label{N:290}

There is no null structure in the bilinear term $\im \left( \phi \overline{\partial \phi} \, \right)$ in the equation for $A$ (the second equation in \eqref{N:100}), but as remarked already, it is the regularity of $\mathbf E, \mathbf B$ which is of interest, not that of $A$, and the equations for $\mathbf E, \mathbf B$ do exhibit null structure, independently of the gauge.

From Maxwell's equations we have $\square \mathbf E = \partial_t \mathbf J + \nabla \rho$ and $\square \mathbf B = - \nabla \times \mathbf J$, where $\rho$ and $\mathbf J$ are given by \eqref{A:8}, hence
\begin{equation}\label{N:292}
\left\{
\begin{aligned}
  \square\mathbf E
  &=
  \im \left( \partial_t\phi\overline{\nabla\phi} - \nabla\phi\overline{\partial_{t}\phi} \right)
  + \partial_t(\mathbf A \abs{\phi}^2) - \nabla (A_0\abs{\phi}^2),
  \\
  \square\mathbf B
  &=
  -\im \left( \nabla\phi\times \overline{\nabla\phi})-\nabla\times
(\mathbf A \abs{\phi}^2 \right).
\end{aligned}
\right.
\end{equation}
Here $\partial_t\phi\overline{\nabla\phi} - \nabla\phi\overline{\partial_{t}\phi}$ and $\nabla\phi\times\overline{\nabla \phi}$ are vectors whose
components consist of the null forms $Q_{0j}(\phi, \overline \phi)$ and $Q_{ij}(\phi, \overline \phi)$, respectively, where $Q_{0j}(u,v)=\partial_tu\partial_jv-\partial_ju\partial_t v$ and $Q_{ij}(u,v)=\partial_iu\partial_jv-\partial_ju\partial_i v$ are Klainerman's null forms. For our purposes, however, it is better to recast the null structure in terms of the splitting $\phi=\phi_+ + \phi_-$, as we did for $P_1$ in \eqref{N:210}. Recalling that $\partial_t\phi = i\angles{\nabla}_m(\phi_+-\phi_-)$, we find
\begin{align}
  \label{N:300}
  \partial_t\phi\overline{\nabla\phi} - \nabla\phi\overline{\partial_{t}\phi}
  &=
  \sum_{\pm_1,\pm_2}
  (\pm_1 1)(\pm_2 1) \mathfrak B_{(\pm_1,\pm_2)}(\phi_{\pm_1}, \phi_{\pm_2}),
  \\
  \label{N:310}
  \nabla\phi\times \overline{\nabla\phi}
  &=
  \sum_{\pm_1, \pm_2} \mathfrak C_{(\pm_1,\pm_2)}(\phi_{\pm_1},\phi_{\pm_2})
\end{align}
where
\begin{align}
  \label{N:320}
  &\mathfrak B_{(\pm_1,\pm_2)}(f,g)
  =
  i \left(\angles{\nabla}_m f \overline{\nabla(\pm_2 g)}
  +
  \nabla(\pm_1 f) \overline {\angles{\nabla}_m g} \right),
  \\
  \label{N:330}
  &\mathfrak C_{(\pm_1,\pm_2)}(f,g)
  =
  \nabla f \times \overline{\nabla g},
\end{align}
and the associated symbols are, respectively,
\begin{align}
  \label{N:340}
  &\mathfrak b_{(\pm_1,\pm_2)}(\eta,\zeta)
  =
  \angles{\eta}_m (\pm_2\zeta)
  -
  \angles{\zeta}_m (\pm_1\eta),
  \\
  \label{N:350}
  &\mathfrak c_{(\pm_1,\pm_2)}(\eta,\zeta)
  =
  \eta \times \zeta,
\end{align}
which satisfy the following estimates, demonstrating the null structure (in the case of $\mathfrak b_{(\pm_1,\pm_2)}$ up to a lower order term due to the positive mass $m$):

\begin{lemma}\label{N:Lemma2} For all nonzero $\eta,\zeta \in \R^3$,
\begin{align}
  \label{N:350}
  \Abs{\mathfrak b_{(\pm_1,\pm_2)}(\eta,\zeta)} &\lesssim m (\abs{\eta} + \abs{\zeta}) + \abs{\eta}\abs{\zeta}\theta(\pm_1\eta,\pm_2\zeta),
  \\
  \label{N:360}
  \Abs{\mathfrak c_{(\pm_1,\pm_2)}(\eta,\zeta)} &\lesssim \abs{\eta}\abs{\zeta}\theta(\pm_1\eta,\pm_2\zeta).
\end{align}
\end{lemma}

\begin{proof}
Since $\abs{\mathfrak c_{(\pm_1,\pm_2)}(\eta,\zeta)}
=
\abs{\eta \times \zeta}
=
\abs{(\pm_1\eta) \times (\pm_2 \zeta)}
=
\abs{\eta}\abs{\zeta} \sin\theta(\pm_1\eta,\pm_2\zeta)$, we get \eqref{N:360}. For \eqref{N:350}, we may assume without loss of generality that $\pm_1=\pm_2=+$, and we write
\begin{align*}
  \angles{\eta}_m \zeta
  -
  \angles{\zeta}_m \eta
  &=
  \left(\sqrt{m^2+\abs{\eta}^2} - \abs{\eta}\right) \zeta
  + \abs{\eta}\zeta
  - \left(\sqrt{m^2+\abs{\zeta}^2} - \abs{\zeta}\right) \eta
  - \abs{\zeta}\eta.
\end{align*}
Taking absolute values, and noting that the two terms in parentheses are $O(m)$, we reduce to checking that $\bigabs{ \abs{\eta}\zeta - \abs{\zeta}\eta} \lesssim \abs{\eta}\abs{\zeta}\theta(\eta,\zeta)$. But this is clear, since the square of the left hand side is $2\abs{\eta}^2\abs{\zeta}^2 ( 1- \cos\theta(\eta,\zeta) )$.
\end{proof}

\section{Local well-posedness, part I}\label{L}

We first prove local existence for \eqref{N:205}, with unknowns $(\phi_+,\phi_-,A_+,A_-)$, so we need to express $\mathcal M(A,\phi)$ and $\mathcal N(A,\phi)$, defined as in \eqref{N:100}, entirely in terms of $(\phi_+,\phi_-,A_+,A_-)$ and their spatial derivatives (no time derivatives should appear, since \eqref{N:205} is first order in time). This can be achieved by writing $\phi=\phi_+ + \phi_-$ and $A=A_+ + A_-$, and replacing $\partial_t \phi$ by $i\angles{\nabla}_m(\phi_+-\phi_-)$. However, for the term $2i A_\mu \partial^\mu \phi$, appearing in $\mathcal M(A,\phi)$, we do something less obvious: Write it as $2i(P_1+P_2)$, where $P_2$ is the null form in \eqref{N:140}, and $P_1$ is expressed as in \eqref{N:210}; note that the latter relies explicitly on the assumption that we are in Lorenz gauge, and this assumption will eventually have to be justified (see the next section) when passing to the true M-K-G system. But in this section we consider only the resulting system of equations for $(\phi_+,\phi_-,A_+,A_-)$, which reads:
\begin{equation}\label{L:100}
\left\{
\begin{aligned}
  \left(i\partial_t \pm \angles{\nabla}_m\right) \phi_\pm
  &= - (\pm 2\angles{\nabla}_m)^{-1}
  \mathfrak M(\phi_+,\phi_-,A_+,A_-)
  \\
  \left(i\partial_t \pm \abs{\nabla}\right) A_\pm
  &=
  - (\pm 2\abs{\nabla})^{-1} \mathfrak N(\phi_+,\phi_-,A_+,A_-),
\end{aligned}
\right.
\end{equation}
where
\begin{equation}\label{L:110}
\left\{
\begin{aligned}
  \mathfrak M(\phi_+,\phi_-,A_+,A_-)
  &= 2 \sum_{\pm_1,\pm_2}
  \pm_2 \mathfrak A_{(\pm_1,\pm_2)}(A_{0,\pm_1}, \phi_{\pm_2})
  \\
  &\qquad\quad
  +
  2i( \Delta^{-1} \nabla [(\nabla \times \mathbf A)_l] \times \nabla \phi)^l
  + 
  A_\mu A^\mu \phi,
  \\
  \mathfrak N_0(\phi_+,\phi_-,A_+,A_-)
  &=
  \im \left[\phi i\angles{\nabla}_m
  \left( \overline{\phi_+}-\overline{\phi_-} \right) \right]
  -  A_0 \abs{\phi}^2,
  \\
  \mathfrak N_j(\phi_+,\phi_-,A_+,A_-)
  &=
  -
  \im \left( \phi\overline{ \partial_j \phi} \right)
  - \mathbf{A} \abs{\phi}^2
  \qquad \text{for $j=1,2,3$}.
\end{aligned}
\right.
\end{equation}
Here it is understood that $\phi = \phi_+ + \phi_-$ and $A = A_+ + A_-$, by definition. The initial data are
\begin{equation}\label{L:120}
\left\{
\begin{aligned}
  \phi_\pm\init
  &=
  \frac12 \left( \phi_0 \pm (i\angles{\nabla}_m)^{-1} \phi_1 \right)  \in H^1,
  \\
  A_\pm\init
  &=
  \frac12\left( a \pm (i\abs{\nabla})^{-1} \dot a \right)\in \dot H^1.
\end{aligned}
\right.
\end{equation}
We split $A_\pm$ into its homogeneous and inhomogeneous parts:
$$
  A_\pm = A_\pm^{(0)} + A_\pm^{\text{inh.}},
$$
where $\left(i\partial_t \pm \abs{\nabla}\right) A_\pm^{(0)} = 0$ with initial data as in \eqref{L:120}, whereas $A_\pm^{\text{inh.}}$ satisfies the second equation in \eqref{L:100}, but with zero initial data.

We shall use a contraction argument to prove local existence and uniqueness, and to this end we introduce $X^{s,b}$ spaces associated to the operators $\left(i\partial_t \pm \angles{\nabla}_m\right)$ and $\left(i\partial_t \pm \abs{\nabla}\right)$, whose symbols are, respectively, $-\tau \pm \angles{\xi}_m$ and $-\tau \pm \abs{\xi}$. However, these symbols are comparable in the sense that
\begin{equation}\label{L:130}
  \angles{-\tau \pm \angles{\xi}_m} \sim \angles{-\tau \pm \abs{\xi}},
\end{equation}
where $\angles{\cdot} = \sqrt{1+\abs{\cdot}^2}$, hence the $X^{s,b}$ spaces corresponding to the two operators are in fact identical. Note that \eqref{L:130} holds since $\angles{\xi}_m = \abs{\xi} + (\sqrt{m^2+\abs{\xi}^2} - \abs{\xi}) = \abs{\xi} + O(m)$.

\begin{definition}\label{L:Def}
For $s,b \in \R$, let $X_\pm^{s,b}$ be the completion of the Schwartz space $\mathcal{S}(\R^{1+3})$ with respect to the norm
$$
  \norm{u}_{X_\pm^{s,b}} = \bignorm{ \angles{\xi}^s
  \angles{-\tau\pm\abs{\xi}}^b \,\widehat u(\tau,\xi)}_{L^2_{\tau,\xi}},
$$
where $\widehat u(\tau,\xi)$ denotes the space-time Fourier transform of $u(t,x)$.
\end{definition}

For $T > 0$, let $X_\pm^{s,b}(S_T)$ denote the restriction space to $S_T = (-T,T) \times \R^3$. We recall the fact that
$$
  X_\pm^{s,b}(S_T) \hookrightarrow C([-T,T]; H^s)
  \qquad
  \text{for $b > \frac12$,}
$$
where $\hookrightarrow$ stands for continuous inclusion. Moreover, it is well-known that the linear initial value problem
$$
  \left(-i\partial_t \pm \abs{\nabla}\right) u = F,
  \qquad
  u\init=u_0,
$$
for given $F \in X_\pm^{s,b-1}(S_T)$ and $u_0 \in H^s$, any $s \in \R$ and $b > \frac12$, has a unique solution $u \in X_\pm^{s,b}(S_T)$, satisfying
\begin{equation}\label{L:140}
  \norm{u}_{X_\pm^{s,b}(S_T)} \le C_b(T) \left( \norm{u_0}_{H^s} + \norm{F}_{X_\pm^{s,b-1}(S_T)} \right),
\end{equation}
where $C_b$ is bounded as $T \to 0$. A proof of this, which applies to $X^{s,b}$ spaces in general, can be found in \cite{Kenig:1994}. Moreover (see, e.g., \cite{Selberg:2007d}), at the cost of a small loss of regularity for $F$, namely by replacing $\norm{F}_{X_\pm^{s,b-1}(S_T)}$ above by $\norm{F}_{X_\pm^{s,b-1+\varepsilon}(S_T)}$ for some small $\varepsilon > 0$, one can ensure that $C_b(T) = O(T^\varepsilon)$ as $T \to 0$, allowing one to deal with large initial data in a contraction mapping setup. Finally, in view of the estimate \eqref{L:130}, the same statements hold with $\abs{\nabla}$ replaced by $\angles{\nabla}_m$, i.e., for the equation $\left(-i\partial_t \pm \angles{\nabla}_m\right) u = F$.

We shall prove the following local existence and uniqueness result.

\begin{theorem}\label{L:Thm1} Given any initial condition
$$
  (\phi_+,\phi_-,A_+,A_-) \init \in H^1 \times H^1 \times \dot H^1 \times \dot H^1,
$$
there exists a $T > 0$, depending continuously on the data norm
$$
  N_0 = \norm{(\phi_+,\phi_-,A_+,A_-)\init}_{H^1 \times H^1 \times \dot H^1 \times \dot H^1},
$$
and there exists a solution $(\phi_+,\phi_-,A_+,A_-)$ of \eqref{L:100} on $S_T = (-T,T) \times \R^3$ satisfying the given initial condition. The solution has the regularity
\begin{align*}
  \phi_\pm &\in X_\pm^{1,b}(S_T) \subset C([-T,T],H^1),
  \\
  A_\pm^{(0)} &\in \abs{\nabla}^{-1} X_\pm^{0,\beta}(S_T)
  \subset C([-T,T],\dot H^1),
  \\
  A_\pm^{\text{inh.}} &\in X_\pm^{1-\delta,\beta}(S_T)
  \subset C([-T,T], H^{1-\delta}),
\end{align*}
for some $b,\beta > \frac12$ and any $\delta > 0$. The solution is unique in this regularity class, for $\delta > 0$ small enough.
\end{theorem}

We use the obvious iteration scheme for \eqref{L:100}, and denote the sequence of iterates by $\bigl\{\phi_\pm^{(n)}, A_\pm^{(n)}\bigr\}_{n=0}^\infty$. The iterates $\phi_\pm^{(0)}, A_\pm^{(0)}$ are just the solutions of the homogeneous equations $\left(i\partial_t \pm \angles{\nabla}_m\right) \phi_\pm^{(0)} = 0$ and $\left(i\partial_t \pm \abs{\nabla}\right) A_\pm^{(0)} = 0$, with the given initial data, hence by the linear theory discussed after Definition \ref{L:Def}, we have $\phi^{(0)} \in X_\pm^{1,b}(S_T)$ and $A_\pm^{(0)} \in \abs{\nabla}^{-1}X_\pm^{0,\beta}(S_T)$ for every $T > 0$ and every $b,\beta > \frac12$. Moreover, assuming henceforth $T \le 1$, we have
$$
  \bignorm{\phi_\pm^{(0)}}_{X_\pm^{1,b}(S_T)} \lesssim N_0,
  \qquad
  \bignorm{\abs{\nabla}A_\pm^{(0)}}_{X_\pm^{0,\beta}(S_T)} \lesssim N_0,
$$
where $N_0$ is the initial data norm.

The subsequent iterates $\phi_\pm^{(n)}, A_\pm^{(n)}$, for $n \ge 1$, are defined by solving \eqref{L:100} with the superscripts $(n)$ and $(n-1)$ inserted on the left and right hand sides, respectively, and with the initial conditions given in Theorem \ref{L:Thm1}. For $A_\pm^{(n)}$ it is crucial that we split into the homogeneous and inhomogeneous parts:
$$
  A_\pm^{(n)} = A_\pm^{(0)} + A_\pm^{\text{inh.},(n)},
$$
where $A_\pm^{\text{inh.},(n)}$ is defined like $A_\pm^{(n)}$, but with zero initial data.

Applying the linear theory discussed after Definition \ref{L:Def}, then by a standard argument, Theorem \ref{L:Thm1} reduces to proving, for $\varepsilon > 0$ small enough, the nonlinear estimates
\begin{align}
  \label{L:200}
  \norm{\angles{\nabla}_m^{-1}
  \mathfrak M(\phi_+,\phi_-,A_+,A_-)}_{X_\pm^{1,b-1+\varepsilon}}
  \lesssim R^2 + R^3,
  \\
  \label{L:210}
  \norm{\abs{\nabla}^{-1}
  \mathfrak N(\phi_+,\phi_-,A_+,A_-)}_{X_\pm^{1-\delta,\beta-1+\varepsilon}}
  \lesssim R^2 + R^3,
\end{align}
for all $\phi_+,\phi_-,A_+,A_- \in \mathcal S(\R^{1+3})$, where
$$
  R = \sum_\pm \left[ \norm{\phi_\pm}_{X_\pm^{1,b}} + \min\left(\norm{\abs{\nabla}A_\pm}_{X_\pm^{0,\beta}},\norm{A_\pm}_{X_\pm^{1-\delta,\beta}}\right) \right].
$$
In particular, note that once we have proved this, then the analogous estimates with the norms restricted to $S_T$ follow immediately.

In addition to $X_\pm^{s,b}$, we shall make use of the wave-Sobolev space $H^{s,b}$, defined as the completion of $\mathcal S(\R^{1+3})$ with respect to the norm
$$
  \norm{u}_{H^{s,b}} = \bignorm{ \angles{\xi}^s
  \angles{\abs{\tau}-\abs{\xi}}^b \,\widehat u(\tau,\xi)}_{L^2_{\tau,\xi}}.
$$
Clearly,
\begin{alignat}{2}
  \label{L:220}
  \norm{u}_{H^{s,b}}
  &\le \norm{u}_{X_\pm^{s,b}}& &\qquad\text{for $b \ge 0$},
  \\
  \label{L:222}
  \norm{u}_{X_\pm^{s,b}}
  &\le \norm{u}_{H^{s,b}}& &\qquad\text{for $b \le 0$},
\end{alignat}
allowing us, in particular, to pass from estimates in $X_\pm^{s,b}$ to corresponding estimates in $H^{s,b}$, once the null structure (which depends on the signs) has been exploited.

In the following subsections we prove \eqref{L:200} and \eqref{L:210}, with
$$
  b = \frac12 + \varepsilon,
  \qquad
  \beta = 1-\delta,
$$
where $\varepsilon, \delta > 0$ are understood to be sufficiently small. That is, if we say that some estimate holds, we mean that it holds for $\varepsilon,\delta > 0$ small enough.
 
\subsection{The term $ \mathfrak A_{(\pm_1,\pm_2)}(A_{0,\pm_1}, \phi_{\pm_2})$}\label{L:224} For this term, \eqref{L:200} can be reduced to (here we use \eqref{L:222})
\begin{align}
  \label{L:230}
  I
  &\lesssim \norm{\abs{\nabla} u}_{X_{\pm_1}^{0,1-\delta}}
  \norm{v}_{X_{\pm_2}^{1,\frac12+\varepsilon}},
  \\
  \label{L:240}
  I
  &\lesssim
  \norm{u}_{X_{\pm_1}^{1-\delta,1-\delta}}
  \norm{v}_{X_{\pm_2}^{1,\frac12+\varepsilon}},
\end{align}
where
$$
  I
  =
  \norm{
  \iint
  \frac{\mathfrak a_{(\pm_1,\pm_2)}(\eta,\xi-\eta)}{\angles{\abs{\tau}-\abs{\xi}}^{\frac12-2\varepsilon}}
  \widehat u(\lambda,\eta)
  \widehat v(\tau-\lambda,\xi-\eta) \d\lambda \d\eta
  }_{L^2_{\tau,\xi}}
$$
and the symbol is given by \eqref{N:240}. We may assume without loss of generality that $\widehat u, \widehat v \ge 0$, and we put the symbol in absolute value. Now we apply Lemma \ref{N:Lemma1}, and estimate the angle using the following well-known fact related to the geometry of the null cone (a proof can be found, e.g., in \cite{Selberg:2008c}; see Lemma 2.1 there):

\begin{lemma}\label{L:Lemma}
Assume $s \in [0,\frac12]$. Then for all signs $(\pm_1,\pm_2)$, all $\lambda,\mu \in \R$ and all nonzero $\eta,\zeta \in \R^3$,
$$
  \theta(\pm_1\eta,\pm_2\zeta)
  \lesssim
  \left(
  \frac{\angles{\abs{\lambda+\mu}-\abs{\eta+\zeta}}}
  {\min(\angles{\eta},\angles{\zeta})}
  \right)^{s}
  +
  \left(
  \frac{\angles{-\lambda\pm_1\abs{\eta}}
  + \angles{-\mu\pm_2\abs{\zeta}}}
  {\min(\angles{\eta},\angles{\zeta})}
  \right)^{\frac12}.
$$
\end{lemma}

Applying this with $s=\frac12-2\varepsilon$, and using Lemma \ref{N:Lemma1}, we get $I \lesssim I_1 + I_2 + I_3 + I_4$, where
\begin{align*}
  I_1
  &=
  \norm{
  \iint
  \widehat u(\lambda,\eta)
  \widehat v(\tau-\lambda,\xi-\eta)
  \d\lambda \d\eta
  }_{L^2_{\tau,\xi}},
  \\
  I_2
  &=
  \norm{
  \iint
  \frac{\widehat u(\lambda,\eta)
  \abs{\xi-\eta}\widehat v(\tau-\lambda,\xi-\eta) }
  {\min(\angles{\eta},\angles{\xi-\eta})^{\frac12-2\varepsilon}}
  \d\lambda \d\eta
  }_{L^2_{\tau,\xi}},
  \\
  I_3
  &=
  \norm{
  \iint
  \frac{\angles{-\lambda\pm_1\abs{\eta}}^{\frac12}\widehat u(\lambda,\eta)
  \abs{\xi-\eta}\widehat v(\tau-\lambda,\xi-\eta) }
  {\angles{\abs{\tau}-\abs{\xi}}^{\frac12-2\varepsilon}
  \min(\angles{\eta},\angles{\xi-\eta})^{\frac12}}
  \d\lambda \d\eta
  }_{L^2_{\tau,\xi}}
  \\
  I_4
  &=
  \norm{
  \iint
  \frac{\widehat u(\lambda,\eta)
  \angles{-(\tau-\lambda)\pm_2\abs{\xi-\eta}}^{\frac12}
  \abs{\xi-\eta}\widehat v(\tau-\lambda,\xi-\eta) }
  {\angles{\abs{\tau}-\abs{\xi}}^{\frac12-2\varepsilon}
  \min(\angles{\eta},\angles{\xi-\eta})^{\frac12}}
  \d\lambda \d\eta
  }_{L^2_{\tau,\xi}}
\end{align*}
Using also \eqref{L:220}, we can thus reduce \eqref{L:240} to the estimates
\begin{equation}\label{L:242}
\left\{
\begin{aligned}
  \norm{uv}_{L^2}
  &\lesssim
  \norm{u}_{H^{1-\delta,1-\delta}}
  \norm{v}_{H^{1,\frac12+\varepsilon}}
  \\
  \norm{uv}_{L^2}
  &\lesssim
  \norm{u}_{H^{\frac32-2\varepsilon-\delta,1-\delta}}
  \norm{v}_{H^{0,\frac12+\varepsilon}}
  \\
  \norm{uv}_{L^2}
  &\lesssim
  \norm{u}_{H^{1-\delta,1-\delta}}
  \norm{v}_{H^{\frac12-2\varepsilon,\frac12+\varepsilon}}
  \\
  \norm{uv}_{L^2}
  &\lesssim
  \norm{u}_{H^{\frac32-2\varepsilon-\delta,\frac12-\delta}}
  \norm{v}_{H^{0,\frac12+\varepsilon}}
  \\
  \norm{uv}_{L^2}
  &\lesssim
  \norm{u}_{H^{1-\delta,\frac12-\delta}}
  \norm{v}_{H^{\frac12-2\varepsilon,\frac12+\varepsilon}}
  \\
  \norm{uv}_{H^{0,-\frac12+2\varepsilon}}
  &\lesssim
  \norm{u}_{H^{\frac32-2\varepsilon-\delta,1-\delta}}
  \norm{v}_{L^2}
  \\
  \norm{uv}_{H^{0,-\frac12+2\varepsilon}}
  &\lesssim
  \norm{u}_{H^{1-\delta,1-\delta}}
  \norm{v}_{H^{\frac12-2\varepsilon,0}}.
\end{aligned}
\right.
\end{equation}
These are of the general form
\begin{equation}\label{L:250}
   \norm{uv}_{H^{-s_0,-b_0}} \le C \norm{u}_{H^{s_1,b_1}}\norm{v}_{H^{s_2,b_2}},
\end{equation}
where $s_0,s_1,s_2,b_0,b_1,b_2 \in \R$.

\begin{definition}
If \eqref{L:250} holds for all $u,v \in \mathcal S(\R^{1+3})$, we say that the exponent matrix
$\Hproduct{s_0 & s_1 & s_2 \\ b_0 & b_1 & b_2}$ is a \emph{product}.
\end{definition}

Many estimates of the form \eqref{L:250} have appeared in the literature, but for a long time no systematic effort was made to determine necessary and sufficient conditions on $\Hproduct{s_0 & s_1 & s_2 \\ b_0 & b_1 & b_2}$ for it to be a product. With the recent work \cite{Selberg:2009a}, however, such conditions, sharp up to some endpoint cases, are available to us.

Note that if $\Hproduct{s_0 & s_1 & s_2 \\ b_0 & b_1 & b_2}$ is a product, then so is every permutation of its columns. Using this fact, we see that all the estimates in \eqref{L:242} hold, for $\varepsilon, \delta > 0$ small, by the following result, proved in \cite{Selberg:2009a}, about products of the form  $\Hproduct{s_0 & s_1 & s_2 \\ 0 & b_1 & b_2}$:

\begin{theorem}\label{L:Thm2} Set $b_0=0$ and assume
\begin{align}
  \label{L:301}
  &b_1, b_2 > 0
  \\
  \label{L:302}
  &b_1 + b_2 \ge \frac12
  \\
  \label{L:303}
  &s_0 + s_1 + s_2 \ge 2 - (b_1 + b_2)
  \\
  \label{L:304}
  &s_0 + s_1 + s_2 \ge \frac32 - b_1
  \\
  \label{L:305}
  &s_0 + s_1 + s_2 \ge \frac32 - b_2
  \\
  \label{L:306}
  &s_0 + s_1 + s_2 \ge 1
  \\
  \label{L:307}
  &s_0 + 2(s_1 + s_2) \ge \frac32  \\
  \label{L:308}
  &s_1 + s_2 \ge 0
  \\
  \label{L:309}
  &s_0 + s_2 \ge 0
  \\
  \label{L:310}
  &s_0 + s_1 \ge 0,
\end{align}
as well as the exceptions:
\begin{align}
  \label{L:311a}
  &\parbox[t]{10.8cm}{If $b_1=\frac12$, then~\eqref{L:303} and \eqref{L:305} must be strict.}
    \\
  \label{L:311b}
  &\parbox[t]{10.8cm}{If $b_1=\frac12$, then~\eqref{L:304} and \eqref{L:306} must be strict.}
  \\
  \label{L:312a}
  &\parbox[t]{10.8cm}{If $b_2=\frac12$, then~\eqref{L:303} and \eqref{L:304} must be strict.}
  \\
  \label{L:312b}
  &\parbox[t]{10.8cm}{If $b_2=\frac12$, then~\eqref{L:305} and \eqref{L:306} must be strict.}
  \\
  \label{L:315}
  &\parbox[t]{10.8cm}{If $b_1+b_2=1$, then~\eqref{L:303} and \eqref{L:306} must be strict.}
  \\
  \label{L:316}
  &\parbox[t]{10.8cm}{We require~\eqref{L:307} to be strict if $s_0$ takes one of the values $\frac12$, $\frac32$, $\frac32 - 2b_1$, $\frac32 - 2b_2$ or $\frac52 - 2(b_1+b_2)$.}
  \\ 
  \label{L:317}
  &\parbox[t]{10.8cm}{If one of~\eqref{L:303}--\eqref{L:306} is an equality, then~\eqref{L:308}--\eqref{L:310} must be strict.}
\end{align}
Then $\Hproduct{s_0 & s_1 & s_2 \\ 0 & b_1 & b_2}$ is a product.
\end{theorem}

We have proved \eqref{L:240}, and clearly \eqref{L:230} follows by the same argument except when $u$ has spatial Fourier support in the region $\abs{\xi} \le 1$. But if this is the case, then by the $L^p$ inequality for potentials on $\R^3$ we have
\begin{equation}\label{L:318}
  \norm{u(t)}_{L^\infty} \lesssim \norm{(1-\Delta)^{\frac14+\varepsilon}u(t)}_{L^6} \lesssim \norm{(1-\Delta)^{\frac14+\varepsilon}\abs{\nabla} u(t)}_{L^2} \sim \norm{\abs{\nabla} u(t)}_{L^2},
\end{equation}
so we can simply estimate
$$
  I \lesssim \norm{u \abs{\nabla} v}_{L^2} \lesssim \norm{u}_{L^\infty} \norm{\abs{\nabla} v}_{L^2} \lesssim \norm{\abs{\nabla} u}_{L_t^\infty L_x^2}  \norm{\abs{\nabla} v}_{L^2} 
$$
and use the fact that $X_\pm^{0,\frac12+\varepsilon} \hookrightarrow L_t^\infty L_x^2$. This completes the proof of \eqref{L:230}.

\subsection{The term $( \Delta^{-1} \nabla [(\nabla \times \mathbf A)_l] \times \nabla \phi)^l$} After splitting $\mathbf A = \mathbf A_+ + \mathbf A_-$ and $\phi=\phi_+ + \phi_-$, the symbol satisfies (this follows as is in the proof of \eqref{N:360}) an estimate like the one in Lemma \ref{N:Lemma1}, but without the $O(m)$ term, and with only one power of the angle. Therefore, the estimates in subsection \ref{L:224} apply.

\subsection{The term $A_\mu A^\mu \phi$} For this term, \eqref{L:200} reduces to (here we use again \eqref{L:220} and \eqref{L:222})
\begin{align}
  \label{L:300}
  \norm{u^2 v}_{L^2}
  &\lesssim
  \norm{\abs{\nabla}u}_{H^{0,\frac12+\varepsilon}}^2
  \norm{v}_{H^{1,\frac12+\varepsilon}}
  \\
  \label{L:310}
  \norm{u^2 v}_{L^2}
  &\lesssim
  \norm{u}_{H^{1-\delta,\frac12+\varepsilon}}^2
  \norm{v}_{H^{1,\frac12+\varepsilon}}.
\end{align}
For \eqref{L:300} use H\"older's inequality and the embedding $\dot H^1 \hookrightarrow L^6$ in $x$, as well as $H^{0,\frac12+\varepsilon} \hookrightarrow L_t^\infty L_x^2$. To prove \eqref{L:310}, we use Theorem \ref{L:Thm2} twice, obtaining
$$
  \norm{u^2 v}_{L^2}
  \lesssim
  \norm{u^2}_{H^{\frac12+\varepsilon,0}}
  \norm{v}_{H^{1,\frac12+\varepsilon}}
  \lesssim
  \norm{u}_{H^{1-\delta,\frac12+\varepsilon}}^2
  \norm{v}_{H^{1,\frac12+\varepsilon}}.
$$

\subsection{The terms $\phi_{\pm_1}\angles{\nabla}_m \overline{\phi_{\pm_2}}$ and $\phi\overline{ \partial_j \phi}$} For these terms, \eqref{L:210} reduces to (using again \eqref{L:220} and \eqref{L:222})
\begin{equation}\label{L:320}
  \norm{\abs{\nabla}^{-1}(uv)}_{H^{1-\delta,0}}
  \lesssim
  \norm{u}_{H^{1,\frac12+\varepsilon}}
  \norm{v}_{H^{0,\frac12+\varepsilon}}.
\end{equation}

First, if the spatial Fourier support of $uv$ is contained in the region $\abs{\xi} \le 1$, then we can replace the left hand side by $\norm{\abs{\nabla}^{-1}(uv)}_{L^2}$, and using the $L^p$ inequality for potentials in $x$ we get
$$
  \norm{\abs{\nabla}^{-1}(uv)}_{L^2}
  \lesssim
  \norm{uv}_{L_t^2 L_x^{\frac65}}
  \le
  \norm{u}_{L_t^2 L_x^3}
  \norm{v}_{L_t^\infty L_x^{2}}
  \lesssim
  \norm{u}_{H^{\frac12,0}}
  \norm{v}_{H^{0,\frac12+\varepsilon}}.
$$

If, on the other hand, $\abs{\xi} > 1$ in the spatial Fourier support of $uv$, then we can replace the left hand side of \eqref{L:320} by $\norm{uv}_{H^{-\delta,0}}$, and the desired estimate holds by Theorem \ref{L:Thm2}.

\subsection{The term $A \abs{\phi}^2$} Then \eqref{L:210} reduces to, splitting into $\abs{\xi} \le 1$ and $\abs{\xi} > 1$ as in the previous subsection,
\begin{align}
  \label{L:330}
  \norm{u^2 v}_{L_t^2 L_x^{\frac65}}
  &\lesssim
  \norm{u}_{H^{1,\frac12+\varepsilon}}^2
  \norm{\abs{\nabla} v}_{H^{0,\frac12+\varepsilon}},
  \\
  \label{L:340}
  \norm{u^2 v}_{L_t^2 L_x^{\frac65}}
  &\lesssim
  \norm{u}_{H^{1,\frac12+\varepsilon}}^2
  \norm{v}_{H^{1-\delta,\frac12+\varepsilon}},
  \\
  \label{L:350}
  \norm{u^2 v}_{L^2}
  &\lesssim
  \norm{u}_{H^{1,\frac12+\varepsilon}}^2
  \norm{\abs{\nabla} v}_{H^{0,\frac12+\varepsilon}},
  \\
  \label{L:360}
  \norm{u^2 v}_{L^2}
  &\lesssim
  \norm{u}_{H^{1,\frac12+\varepsilon}}^2
  \norm{v}_{H^{1-\delta,\frac12+\varepsilon}}.
\end{align}

In fact, \eqref{L:330} holds by a large margin: Estimate the left hand side by  the product of $\norm{u}_{L_t^2 L_x^3}$, $\norm{u}_{L_t^\infty L_x^3}$ and $\norm{v}_{L_t^\infty L_x^{6}}$, and use the embeddings $\dot H^{\frac12} \hookrightarrow L_x^3$, $\dot H^1 \hookrightarrow L_x^6$ and $H^{0,\frac12+\varepsilon} \hookrightarrow L_t^\infty L_x^2$ to get
\begin{equation}\label{L:370}
  \norm{u^2 v}_{L_t^2 L_x^{\frac65}}
  \lesssim
  \norm{u}_{H^{\frac12,\frac12+\varepsilon}}
  \norm{u}_{H^{\frac12,\frac12+\varepsilon}}
  \norm{\abs{\nabla} v}_{H^{0,\frac12+\varepsilon}}.
\end{equation}
A similar argument gives \eqref{L:340}. For \eqref{L:350} we estimate by the product of $\norm{u}_{L_t^2 L_x^6}$, $\norm{u}_{L_t^\infty L_x^6}$ and $\norm{v}_{L_t^\infty L_x^{6}}$. Finally, for \eqref{L:360} we use Theorem \ref{L:Thm2} twice, obtaining
$$
  \norm{u^2 v}_{L^2}
  \lesssim
  \norm{u^2}_{H^{\frac12+\varepsilon+\delta,0}}
  \norm{v}_{H^{1-\delta,\frac12+\varepsilon}}
  \lesssim
  \norm{u}_{H^{1,\frac12+\varepsilon}}^2
  \norm{v}_{H^{1-\delta,\frac12+\varepsilon}}.
$$

This concludes the proof of Theorem \ref{L:Thm1}.

\section{Local well-posedness, part II}\label{LL}

The local well-posedness result in Theorem \ref{L:Thm1}, for a modified system, was proved for arbitary data. In this section we prove that if the data have the form \eqref{L:120}, with $\phi_0$, $\phi_1$, $a$ and $\dot a$ as in Theorem \ref{M:Thm}, then we have in fact a solution of the true M-K-G system in Lorenz gauge:

\begin{theorem}\label{LL:Thm}
Consider the solution $(\phi_+,\phi_-,A_+,A_-)$ of \eqref{L:100} on the time-slab $S_T = (-T,T) \times \R^3$, obtained in Theorem \ref{L:Thm1}, but assume now that the initial data are given by \eqref{L:120}, where $\phi_0$, $\phi_1$, $a$ and $\dot a$ are as in Theorem \ref{M:Thm}. Then defining $\phi = \phi_+ + \phi_-$ and $A = A_+ + A_-$, we have a local solution $(\phi,A)$ of \eqref{A:104} on $S_T$, with data $(\phi,\partial_t \phi) \init = (\phi_0,\phi_1)$ and $(A,\partial_t A) \init = (a,\dot a)$.

Moreover, defining $\mathbf E$ and $\mathbf B$ as in \eqref{A:6}, we obtain a local solution with the properties described in Theorem \ref{M:Thm}, but with the time $t$ restricted to $(-T,T)$, and excluding uniqueness.
\end{theorem}

For the proof of uniqueness, see section \ref{U}.

We now prove Theorem \ref{LL:Thm}. Since the right hand sides of the equations for $\phi_+$ and $\phi_-$ in \eqref{L:100} sum to zero, it is clear that $\phi=\phi_+ + \phi_-$ has time derivative equal to $i\angles{\nabla}_m(\phi_+ - \phi_-)$, and similarly $\partial_t A = i\abs{\nabla}(A_+ - A_-)$. Now it follows that $\mathfrak N$ in \eqref{L:110} is the same as $\mathcal N$ in \eqref{N:100}, and since $\square = (i\partial_t + \abs{\nabla})(i\partial_t - \abs{\nabla}) = (i\partial_t - \abs{\nabla})(i\partial_t + \abs{\nabla})$, it follows immediately that $A = A_+ + A_-$ satisfies the second equation in \eqref{N:100}, which is the same as the second equation in \eqref{A:104}. Moreover, it is clear from \eqref{L:120} that $(\phi,\partial_t \phi) \init = (\phi_0,\phi_1)$ and $(A,\partial_t A) \init = (a,\dot a)$.

Next, we show that the Lorenz condition $\partial^\mu A_\mu = 0$ holds. By regularization of the data, and the fact that the solutions obtained in Theorem \ref{L:Thm1} enjoy continuous dependence on the data and persistence of higher regularity, we may assume that $\phi \in C_c^\infty([-T,T] \times \R^3)$ and $\partial A \in C^\infty([-T,T] \times \R^3)$. Now define $u = \partial^\mu A_\mu$. Write $u = u_+ + u_-$, where $u_\pm = -\partial_t A_{0,\pm} + \nabla \cdot \mathbf A_\pm$. Then by the second equation in \eqref{L:100},
$$
  \left(i\partial_t \pm \abs{\nabla}\right) u_\pm
  =
  - (\pm 2\abs{\nabla})^{-1} \mathfrak R(A,\phi),
$$
where
\begin{align*}
  \mathfrak R(A,\phi)
  &=
  \im \left( \phi\angles{\nabla}_m
  \left(
  -\overline{i\partial_t\phi_+}+\overline{i\partial_t\phi_-}
  \right)
  \right)
  + \im \left( \phi\overline{\Delta\phi} \right)
  + \partial^\mu (A_\mu \abs{\phi}^2)
  \\
  &=
  \im \left( \phi\angles{\nabla}_m
  \left(-
  \overline{[i\partial_t+\angles{\nabla}_m]\phi_+}
  +
  \overline{[i\partial_t- \angles{\nabla}_m]\phi_-}
  \right) \right)
  + \partial^\mu (A_\mu \abs{\phi}^2)
  \\
  &=
  \im \left( \phi \,
  \overline{\mathfrak M(\phi_+,\phi_-,A_+,A_-)}
  \right)
  + A_\mu \partial^\mu (\abs{\phi}^2)
  + \abs{\phi}^2 u.
\end{align*}
Here we used $\Delta\phi
= m^2 \phi-\angles{\nabla}_m 
( \angles{\nabla}_m \phi_+ + \angles{\nabla}_m\phi_-)$ in the second step, and in the last step we used the first equation in \eqref{L:100}.

By the definition \eqref{N:220} and the fact that $\partial_t \phi = i\angles{\nabla}_m(\phi_+ - \phi_-)$ and $\partial_t A = i\abs{\nabla}(A_+ - A_-)$, we know that the second equality in \eqref{N:210} is valid. Using this fact and the identity \eqref{N:140}, we find
$$
  \mathfrak M(\phi_+,\phi_-,A_+,A_-)
  = 2i \left( - A_0 \partial_t \phi + \Delta^{-1} \nabla \partial_t A_0 \cdot \nabla \phi
  + \mathbf A^{\text{df}} \cdot \nabla \phi \right)
  +
  A_\mu A^\mu \phi.
$$
But by definition, $\partial_t A_0 = \nabla \cdot \mathbf A - u$. Using also the identity \eqref{N:110}, we get
$$
  \mathfrak M(\phi_+,\phi_-,A_+,A_-)
  = 2i \left( A_\mu \partial^\mu \phi 
  - \Delta^{-1} \nabla u \cdot \nabla \phi
  \right)
  +
  A_\mu A^\mu \phi.
$$
After simplification we then find
$$
  \mathfrak R(A,\phi)
  =
  2 \re (\phi \overline{\nabla \phi}) \cdot \Delta^{-1} \nabla u
  + \abs{\phi}^2 u.
$$
Therefore,
$$
  \square u =
  2 \re (\phi \overline{\nabla \phi}) \cdot \Delta^{-1} \nabla u
  + \abs{\phi}^2 u.
$$
Since this equation is linear in $u$, and since $\phi \in C_c^\infty([-T,T] \times \R^3)$, uniqueness holds. But by the construction of the data $(a,\dot a)$ in Theorem \ref{M:Thm}, $u \init = \partial_t u \init = 0$, hence $u$ vanishes on $S_T$, so the Lorenz condition $\partial^\mu A_\mu = 0$ is indeed verified there.

With this information in hand, it follows immediately that $\mathfrak M$ in \eqref{L:110} is the same as $\mathcal M$ in \eqref{N:100}, and moreover, the first equation in \eqref{N:100} is the same as the first equation in \eqref{A:104}. This concludes the proof of the first part of Theorem \ref{LL:Thm}.

It remains to prove that $\mathbf E$ and $\mathbf B$, defined by \eqref{A:6}, are continuous in time with values in $L^2$, for $t \in [-T,T]$. For this, it suffices to show $\mathbf E, \mathbf B \in H^{0,\frac12+\varepsilon}(S_T)$. By well-known linear estimates (see, e.g., \cite{Selberg:2002b}) this reduces to checking that
\begin{gather}
  \label{LL:100}
  \mathbf E \init, \mathbf B \init \in L^2
  \qquad \text{and} \qquad
  \partial_t \mathbf E \init, \partial_t \mathbf B \init \in H^{-1},
  \\
  \label{LL:110}
  \square E, \square B \in H^{-1,-\frac12+\varepsilon}(S_T).
\end{gather}

The left member of \eqref{LL:100} holds by assumption \eqref{M:210}, whereas to show the right member we use also Maxwell's equations \eqref{A:2}: $\partial_t \mathbf B \init = - \nabla \times \mathbf E \init \in H^{-1}$ and $\partial_t \mathbf E \init = \nabla \times \mathbf B \init - \mathbf J \init \in H^{-1}$, where $\mathbf J \init \in H^{-1}$ follows by estimating the two terms $\im(\phi_0\overline{\nabla\phi_0})$ and $\abs{\phi_0}^2\mathbf a$; the first term can be estimated as in Remark \ref{A:Rem3}, and the cubic term is in fact in $L^2$, in view of the embedding $\dot H^1 \hookrightarrow L_x^6$.

To prove \eqref{LL:110}, we use \eqref{N:292}. For the bilinear terms, we apply \eqref{N:300} and \eqref{N:310}, as well as the null symbol estimates in Lemma \ref{N:Lemma2}. Proceeding as in subsection \ref{L:224}, we then reduce to the following bilinear estimates:
$$
\begin{aligned}
  \norm{uv}_{H^{-1,0}}
  &\lesssim
  \norm{u}_{H^{0,\frac12+\varepsilon}}
  \norm{v}_{H^{1,\frac12+\varepsilon}}
  \\
  \norm{uv}_{H^{-1,0}}
  &\lesssim
  \norm{u}_{H^{\frac12-\varepsilon,\frac12+\varepsilon}}
  \norm{v}_{H^{0,\frac12+\varepsilon}}
  \\
  \norm{uv}_{H^{-1,-\frac12+\varepsilon}}
  &\lesssim
  \norm{u}_{H^{\frac12,0}}
  \norm{v}_{H^{0,\frac12+\varepsilon}}
  \\
  \norm{uv}_{H^{-1,-\frac12+\varepsilon}}
  &\lesssim
  \norm{u}_{H^{\frac12,\frac12+\varepsilon}}
  \norm{v}_{L^2}.
\end{aligned}
$$
all of which hold for small enough $\varepsilon > 0$, by Theorem \ref{L:Thm2}.

Now consider the cubic terms in \eqref{N:292}. The terms $\nabla (A_0\abs{\phi}^2)$ and $\nabla\times (\mathbf A \abs{\phi}^2)$ are covered by \eqref{L:350} and \eqref{L:360}, leaving $\partial_t(\mathbf A\abs{\phi}^2)$. Since $\partial_t \phi = i\angles{\nabla}_m(\phi_+ - \phi_-)$ and $\partial_t A = i\abs{\nabla}(A_+ - A_-)$, we reduce to estimating terms of the schematic form $\phi^2\abs{\nabla} A$ and $A \phi \abs{\nabla} \phi$. But since our norms depend only the absolute value of the Fourier transform, we can reduce (by the triangle inequality in Fourier space) to $\abs{\nabla}(\phi^2 A)$ and $A \phi \abs{\nabla}\phi$; since the first of these is covered by \eqref{L:350} and \eqref{L:360}, it suffices to check the following:
$$
\begin{aligned}
  \norm{uvw}_{H^{-1,-\frac12+\varepsilon}}
  &\lesssim
  \norm{u}_{H^{1-\delta,\frac12+\varepsilon}}
  \norm{v}_{H^{1,\frac12+\varepsilon}}
  \norm{w}_{H^{0,\frac12+\varepsilon}}
  \\
  \norm{uvw}_{H^{-1,0}}
  &\lesssim
  \norm{\abs{\nabla} u}_{H^{0,\frac12+\varepsilon}}
  \norm{v}_{H^{1,\frac12+\varepsilon}}
  \norm{w}_{H^{0,\frac12+\varepsilon}}.
\end{aligned}
$$
The last of these follows from \eqref{L:370}, since $L^{\frac65}_x \hookrightarrow H^{-1}$, whereas the first one follows by two applications of Theorem \ref{L:Thm2}:
$$
  \norm{uvw}_{H^{-1,-\frac12+\varepsilon}}
  \lesssim
  \norm{u}_{H^{1-\delta,\frac12+\varepsilon}}
  \norm{vw}_{H^{-\delta,0}}
  \lesssim
  \norm{u}_{H^{1-\delta,\frac12+\varepsilon}}
  \norm{v}_{H^{1,\frac12+\varepsilon}}
  \norm{w}_{H^{0,\frac12+\varepsilon}}.
$$
This concludes the proof that $\mathbf E,\mathbf B \in C([-T,T];L^2)$.

Finally, we note that by continuous dependence on the data and persistence of higher regularity, the solution $(\phi,\mathbf E, \mathbf B)$  is a limit, in $C([-T,T];H^1 \times L^2 \times L^2)$, of smooth solutions with $C_c^\infty$ data. This further implies convergence in energy norm, since, for small $\delta > 0$,
$$
  \bignorm{D^{(A)}\phi(t)}_{L^2}
  \le \norm{\nabla \phi(t)}_{L^2}
  + C\norm{A(t)}_{\dot H^1 + H^{1-\delta}}
  \norm{\phi(t)}_{H^1}.
$$
In particular, therefore, energy conservation holds for our solutions, since it holds for the smooth solutions.

\section{Global existence}\label{G}

Recall that the time $T$ in Theorem \ref{L:Thm1} only depends on the initial data norm $N_0$. But in view of \eqref{M:190} and \eqref{M:160}, $N_0$ is controlled by the energy:
$$
  N_0 \lesssim \mathcal E(0).
$$
Since the energy is conserved, this means that the time $T$ will stay fixed if we iterate the local result from Theorem \ref{L:Thm1}. So we extend the solution successively to intervals $[(n-1)T,nT]$ for $n=1,2,\dots$, and similarly in the negative time direction. At each step however, we have to make a gauge transformation as in Remark \ref{A:Rem1} in order to satisfy the hypotheses on the initial gauge in Theorem \ref{L:Thm1}. After applying the local existence theorem we reverse the gauge transformation, so that the local solutions fit together continuously. 

To be precise, at the start of the $n$-th iteration step we construct $\chi$ as in \eqref{M:170}, but now with the data taken at time $(n-1)T$, so $a_0$ and $\mathbf a$ in \eqref{M:170} come from the previous iteration step, hence we know from Theorem \ref{L:Thm1} that they belong to $\dot H^1 + H^{1-\delta}$, for $\delta > 0$ small. Therefore,
$$
  \square \chi = 0,
  \qquad
  \chi((n-1)T) \in \dot H^2 + \abs{\nabla}^{-1} H^{1-\delta},
  \qquad
  \partial_t \chi((n-1)T) \in \dot H^1 + H^{1-\delta},
$$
whence
$$
  \partial \chi \in C(\R;\dot H^1 + H^{1-\delta}),
$$
so the gauge transformation \eqref{A:90} preserves the regularity of $\phi$ and $A$, namely $\phi \in C(I;H^1)$ and $A \in C(I;\dot H^1 + H^{1-\delta})$, where $I \subset \R$ is the time interval. Moreover, the same is true of the inverse transformation, since it is obtained by replacing $\chi$ by $-\chi$. In particular, note that when we estimate the $H^1$ norm of $\phi'$ given by \eqref{A:90}, we need the fact that
$$
  \norm{\nabla \chi(t) \phi(t)}_{L^2}
  \lesssim \norm{\nabla\chi(t)}_{L^3+L^6} \norm{\phi(t)}_{H^1}
  \lesssim
  \norm{\nabla \chi(t)}_{H^{1-\delta}+\dot H^1} \norm{\phi(t)}_{\dot H^1}.
$$
Note also that the Lorenz gauge condition is preserved by \eqref{A:90} and its inverse, since $\square \chi = 0$.

This completes the proof of Theorem \ref{M:Thm} up to uniqueness, which we prove in the next section.

\section{Unconditional uniqueness}\label{U}

It suffices to prove the uniqueness on time intervals $[0,T]$ for small $T > 0$. With the same assumptions on the initial data as in Theorem \ref{M:Thm}, suppose that
$$
  \phi \in C\bigl([0,T];H^1\bigr) \cap C^1\bigl([0,T];L^2\bigr),
  \qquad
  \mathbf E, \mathbf B \in C\bigl([0,T];L^2\bigr)
$$
satisfy the M-K-G system \eqref{A:2}--\eqref{A:8} with initial condition \eqref{M:210}, relative to a real-valued 4-potential $A$ such that
$$
  A,\partial_t A \in C\bigl([0,T];\mathcal D'(\R^3)\bigr),
  \qquad
  \partial^\mu A_\mu = 0,
  \qquad
  (A,\partial_t A)\init=(a,\dot a),
$$
and such that
$$
  \mathcal E(t) = \mathcal E(0) \qquad \text{for all $t$}.
$$

We first make an observation concerning the regularity of $A$. By \eqref{A:104},
\begin{equation}\label{U:100}
  \square A = - J,
  \qquad
  (A,\partial_t A) \init = (a,\dot a) \in \dot H^1 \times L^2.
\end{equation}
But $\mathcal E(t) = \mathcal E(0) < \infty$ implies $D^{(A)}\phi \in L_t^\infty L_x^2$, hence $J \in L_t^\infty L_x^{\frac32}$, which guarantees uniqueness of the solution of \eqref{U:100} in the class $A,\partial_t A \in C\bigl([0,T];\mathcal D'(\R^3)\bigr)$ . Splitting into homogeneous and inhomogeneous parts, $A = A^{(0)} + A^{\text{inh.}}$, it then follows from the energy inequality for the wave equation, and Sobolev embedding, that \begin{align*}
  A^{(0)} &\in C([0,T];\dot H^1),&
  \partial_t A^{(0)} &\in C([0,T];L^2),
  \\
  A^{\text{inh.}} &\in C([0,T];H^{\frac12}),&
  \partial_t A^{\text{inh.}} &\in C([0,T];H^{-\frac12}).
\end{align*}

Now we split $\phi = \phi_+ + \phi_-$ and $A = A_+ + A_-$ using the definition \eqref{N:200}. Clearly, $\phi_+$ and $\phi_-$ belong to $C([0,T];H^1)$, whereas $A_+$ and $A_-$ are well-defined in $C([0,T];\dot H^1 + H^{\frac12} + \abs{\nabla}^{-1} H^{-\frac12})$, hence in $C([0,T];L^6 + H^{\frac12})$, since $\dot H^1 \hookrightarrow L^6$ and $\abs{\nabla}^{-1} H^{-\frac12} \hookrightarrow L^6 + H^{\frac12}$ (cf.\ the remarks at the end of section \ref{A}).

Now it follows that $(\phi_+,\phi_-,A_+,A_-)$ satisfies \eqref{N:205}, or equivalently \eqref{L:100}, on $S_T = (0,T) \times \R^3$, with initial data \eqref{L:120}. We claim that
\begin{align}
  \label{U:110}
  \bignorm{\phi_\pm^{(0)}}_{X_\pm^{1,1-\varepsilon}(S_T)}
  +
  \bignorm{\abs{\nabla} A_\pm^{(0)}}_{X_\pm^{0,1-\varepsilon}(S_T)}
  &\le C,
  \\
  \label{U:120}
  \norm{\phi_\pm^{\text{inh.}}}_{X_\pm^{1-\varepsilon,1-\varepsilon}(S_T)}
  +
  \norm{A_\pm^{\text{inh.}}}_{X_\pm^{1-\varepsilon,1-\varepsilon}(S_T)}
  &\le C,
\end{align}
for sufficiently small $\varepsilon > 0$, where $C$ is a constant that does not depend on $T$ (for $T$ small). Note that \eqref{U:110} follows from ~\eqref{L:140}; \eqref{U:120} is proved in subsections \ref{U:155}--\ref{U:218}.

Next, if $(\widetilde \phi, \widetilde A)$ is another solution with the same initial data (so the homogeneous parts are the same for the two solutions), and in the same regularity class, we introduce the quantity
\begin{equation}\label{U:142}
  \Delta(T) = \sum_{\pm} \left( \norm{\phi_\pm^{\text{inh.}} - \widetilde\phi_\pm^{\text{inh.}}}_{X_\pm^{1-\varepsilon,\frac12+\varepsilon}(S_T)}
  +
  \norm{A_\pm^{\text{inh.}} - \widetilde A_\pm^{\text{inh.}}}_{X_\pm^{1-\delta,1-\delta}(S_T)} \right).
\end{equation}
We claim that for $0 < \varepsilon \ll \delta \ll 1$,
\begin{equation}\label{U:144}
  \Delta(T)
  \le CT^\varepsilon \Delta(T),
\end{equation}
where $C$ is independent of $T > 0$ small. Granting this claim (see subsection \ref{U:300} for its proof), then for $T > 0$ small enough we get $\Delta(T) \le \frac12 \Delta(T)$, hence $\Delta(T) = 0$, and this proves the uniqueness.

To prove \eqref{U:120}, we start from the known facts $\phi_\pm \in L_t^\infty H^1$ and $D^{(A)}\phi \in L_t^\infty L_x^2$, and then we use the structure of the equations \eqref{L:100} (or equivalently \eqref{N:205}) to successively improve the regularity. To this end, we need the following:

\begin{lemma}\label{U:Lemma} Suppose $2 < q \le \infty$ and $2 \le r < \infty$ satisfy $\frac12 \le \frac{1}{q} + \frac{1}{r} \le 1$. Then
$$
  \norm{u}_{L_t^q L_x^r}
  \lesssim
  \bignorm{\abs{\nabla}^{1-\frac{2}{r}}u}_{H^{0,1 - (\frac{1}{q} + \frac{1}{r}) + \gamma}}
$$
holds for any $\gamma > 0$.
\end{lemma}

\begin{proof} Define $\theta = 2 - 2(\frac{1}{q} + \frac{1}{r})$ and $\varepsilon = 1 - \frac{\frac12 - \frac{1}{r}}{1 - (\frac{1}{q} + \frac{1}{r})}$. Then $\theta \in [0,1]$ and $\varepsilon \in (0,1]$, and we have
$$
  \frac{1}{q} = \frac{\theta(1-\varepsilon)}{2} + \frac{1-\theta}{2},
  \qquad
  \frac{1}{r} = \frac{\theta\varepsilon}{2} + \frac{1-\theta}{2}.
$$
By the Strichartz type estimates for the homogeneous 3d wave equation, and the transfer principle, we have (see, e.g., \cite{Selberg:2002b}), for any $\lambda > \frac12$,
$$
  \norm{u}_{L_t^{\frac{2}{1-\varepsilon}}L_x^{\frac{2}{\varepsilon}}}
  \lesssim
  \norm{\abs{\nabla}^{1-\varepsilon}u}_{H^{0,\lambda}}.
$$
Interpolating with $\norm{u}_{L_t^2 L_x^2}
= \norm{u}_{H^{0,0}}$, we get $\norm{u}_{L_t^q L_x^r}\lesssim \norm{\abs{\nabla}^{\theta(1-\varepsilon)}u}_{H^{0,\theta \lambda}}$,
and this gives the desired conclusion, since $\theta(1-\varepsilon) = 1-\frac{2}{r}$.
\end{proof}

Combining this lemma with the $L^p$ inequality for potentials (see \cite[Ch.~V, Thm.~1]{Stein:1970}), we get also, for $q,r$ as in the lemma,
$$
  \norm{u}_{L_t^q L_x^{3r}}
  \lesssim
  \bignorm{\abs{\nabla}u}_{H^{0,1 - (\frac{1}{q} + \frac{1}{r}) + \gamma}},
$$
and together with \eqref{U:110} and \eqref{L:222} this implies
\begin{equation}\label{U:150}
  A_\pm^{(0)} \in L_t^{\frac{2r}{r-2}} L_x^{3r}(S_T)
  \qquad \text{for all $2 \le r < \infty$},
\end{equation}
a fact we shall make use of below.

\subsection{First estimate for $A_\pm^{\text{inh.}}$}\label{U:155}

Applying \eqref{L:140} to the second equation in \eqref{N:205}, where $\mathcal N(A,\phi) = - \im \bigl( \phi \overline{D^{(A)}\phi} \bigr)$ (by \eqref{N:100}), we find that
\begin{equation}\label{U:160}
  \norm{A_\pm^{\text{inh.}}}_{X_\pm^{s,b}(S_T)}
  \lesssim
  \bignorm{\phi \overline{D^{(A)}\phi} }_{X_\pm^{s-1,b-1}(S_T)}
\end{equation}
for any $s \in \R$ and $b > \frac12$, which are still to be determined. On the other hand, we observe that
\begin{equation}\label{U:170}
  \bignorm{\phi \overline{D^{(A)}\phi} }_{L_t^p L_x^{\frac32}(S_T)}
  \le
  \norm{\phi}_{L_t^\infty L_x^6(S_T)}
  \bignorm{D^{(A)}\phi}_{L_t^\infty L_x^2(S_T)} < \infty,
\end{equation}
for any $1 \le p \le \infty$ (recall that implicit constants may depend on $T$, which is fixed). Note that the right hand side is bounded since $\phi \in L_t^\infty H^1$ and $D^{(A)}\phi \in L_t^\infty L_x^2$.

But by Lemma \ref{U:Lemma},
\begin{equation}\label{U:175}
  \norm{u}_{L_t^{\frac{2}{1-\gamma}} L_x^3}
  \lesssim
  \norm{u}_{H^{\frac13,\frac16 + \gamma}}
  \qquad (0 < \gamma \le 1).
\end{equation}
Therefore, by duality and \eqref{L:222},
\begin{equation}\label{U:180}
  \bignorm{\phi \overline{D^{(A)}\phi} }_{X_\pm^{s-1,b-1}(S_T)}
  \lesssim
  \bignorm{\phi \overline{D^{(A)}\phi} }_{L_t^{\frac{2}{1+\gamma}}L_x^{\frac32}(S_T)}
\end{equation}
holds with $s=\frac23$ and $b=\frac56 - \gamma$, for any $\gamma > 0$. From \eqref{U:160}--\eqref{U:180} we conclude:
\begin{equation}\label{U:190}
  A_\pm^{\text{inh.}} \in X_\pm^{\frac23,\frac56-\gamma}(S_T)
  \qquad (\forall \gamma > 0).
\end{equation}

\subsection{First estimate for $\phi_\pm^{\text{inh.}}$} We claim that
\begin{equation}\label{U:200}
  \phi_\pm^{\text{inh.}} \in X_\pm^{\frac23,\frac56-\gamma}(S_T)
  \qquad (\forall \gamma > 0).
\end{equation}
Applying \eqref{L:140} to the first equation in \eqref{N:205}, where $\mathcal M(A,\phi)$ is defined as in \eqref{N:100}, we reduce to proving that $A\partial\phi$ and $A^2\phi$ belong to $X_\pm^{-\frac13,-\frac16-\gamma}(S_T)$. But by the dual of \eqref{U:175}, it suffices to prove that they belong to $L_t^{\frac{2}{1+\gamma}} L_x^\frac32(S_T)$.

By Lemma \ref{U:Lemma} and \eqref{U:190}, $A_\pm^{\text{inh.}} \in L_t^3 L_x^6(S_T)$, and $A_\pm^{(0)} \in L_t^\infty L_x^6(S_T)$ by \eqref{U:150}, so
$$
  \norm{A\partial\phi}_{L_t^{\frac{2}{1+\gamma}} L_x^{\frac32}(S_T)}
  \lesssim
  \norm{A}_{L_t^3 L_x^6(S_T)}
  \norm{\partial\phi}_{L_t^\infty L_x^2(S_T)}
  < \infty.
$$

For $A^2\phi$, it suffices to estimate separately $(A^{(0)})^2 \phi$ and $(A^{\text{inh.}})^2 \phi$. For the former we write
$$
  \bignorm{(A^{(0)})^2\phi}_{L_t^{\frac{2}{1+\gamma}} L_x^{\frac32}(S_T)}
  \lesssim
  \bignorm{A^{(0)}}_{L_t^\infty L_x^6(S_T)}^2
  \norm{\phi}_{L_t^\infty L_x^3(S_T)}
  < \infty.
$$
On the other hand, by Lemma \ref{U:Lemma} and \eqref{U:190}, we also have $A_\pm^{\text{inh.}} \in L_t^4 L_x^4(S_T)$, hence
$$
  \norm{(A^{\text{inh.}})^2\phi}_{L_t^{\frac{2}{1+\gamma}} L_x^{\frac32}(S_T)}
  \lesssim
  \bignorm{A^{\text{inh.}}}_{L_t^4 L_x^4(S_T)}^2
  \norm{\phi}_{L_t^\infty L_x^6(S_T)}
  < \infty.
$$
and this completes the proof of \eqref{U:200}.

\subsection{Inductive estimates}\label{U:218} We claim that, for $m=1,2,\dots$,
\begin{equation}\label{U:220}
  A_\pm^{\text{inh.}}, \phi_\pm^{\text{inh.}} \in X_\pm^{s(m),b(m)-\gamma}(S_T) \qquad (\forall \gamma > 0),
\end{equation}
where
$$
  s(m) = \frac{2^m}{2^m+1},
  \qquad
  b(m) = \frac{2^{m+1}+1}{2^{m+1}+2}.
$$
For $m=1$, \eqref{U:220} holds by \eqref{U:190} and \eqref{U:200}.

We shall prove that if \eqref{U:220} holds for some $m \ge 1$, then it holds for $m+1$ also.

Interpolating \eqref{U:220} with $\phi_\pm^{\text{inh.}} \in X_\pm^{1,0}(S_T) = L_t^2 H^1(S_T)$, we get
$$
  \phi_\pm^{\text{inh.}} \in X_\pm^{\theta + (1-\theta)s(m),(1-\theta)b(m)-\varepsilon}(S_T)
$$
for all $0 \le \theta \le 1$ and $\varepsilon > 0$. Take $\theta = \frac{2^m}{2^{m+1}+1}$ to obtain
\begin{equation}\label{U:240}
  \phi_\pm^{\text{inh.}} \in X_\pm^{s(m+1),\frac12-\varepsilon}(S_T).
\end{equation}
But by Lemma \ref{U:Lemma},
$$
  \norm{u}_{L_t^{\frac{2}{s(m+1)+4\varepsilon}} L_x^{\frac{2}{1-s(m+1)}}}
  \lesssim
  \norm{u}_{H^{s(m+1),\frac12 - \varepsilon}}
$$
for all sufficiently small $\varepsilon > 0$, hence
\begin{equation}\label{U:250}
  \phi_\pm^{\text{inh.}} \in L_t^{\frac{2}{s(m+1)+4\varepsilon}} L_x^{\frac{2}{1-s(m+1)}}(S_T).
\end{equation}
Of course, \eqref{U:240} also holds for $\phi_\pm^{(0)}$, hence so does \eqref{U:250}. Therefore,
\begin{equation}\label{U:260}
  \bignorm{\phi \overline{D^{(A)}\phi} }_{L_t^p L_x^{\frac{2}{2-s(m+1)}}(S_T)}
  \le
  \norm{\phi}_{L_t^{\frac{2}{s(m+1)+4\varepsilon}} L_x^{\frac{2}{1-s(m+1)}}(S_T)}
  \bignorm{D^{(A)}\phi}_{L_t^\infty L_x^2(S_T)}
\end{equation}
is finite, for any $1 \le p \le \frac{2}{s(m+1)+4\varepsilon}$. By \eqref{U:160},
\begin{equation}\label{U:270}
  \norm{A_\pm^{\text{inh.}}}_{X_\pm^{s(m+1),b(m+1)-\gamma}(S_T)}
  \lesssim
  \bignorm{\phi \overline{D^{(A)}\phi} }_{X_\pm^{s(m+1)-1,b(m+1)-1-\gamma}(S_T)},
\end{equation}
and Lemma \ref{U:Lemma} implies (to check this, note that $1-b(m)=\frac{1-s(m)}{2}$ for all $m$)
$$
  \norm{u}_{L_t^{\frac{2}{1-\gamma}} L_x^{\frac{2}{s(m+1)}}}
  \lesssim
  \norm{u}_{X_\pm^{1-s(m+1),1-b(m+1)+\gamma}}
  \qquad (0 < \gamma \le 1).
$$
Then by duality,
\begin{equation}\label{U:280}
  \norm{u}_{X_\pm^{s(m+1)-1,b(m+1)-1-\gamma}(S_T)}
  \lesssim
  \norm{u}_{L_t^{\frac{2}{1+\gamma}} L_x^{\frac{2}{2-s(m+1)}}(S_T)},
\end{equation}
We conclude from \eqref{U:260}--\eqref{U:280} that
\begin{equation}\label{U:290}
  A_\pm^{\text{inh.}} \in X_\pm^{s(m+1),b(m+1)-\gamma}(S_T).
\end{equation}

It remains to prove $\phi_\pm^{\text{inh.}} \in X_\pm^{s(m+1),b(m+1)-\gamma}(S_T)$. Applying \eqref{L:140} to the first equation in \eqref{N:205}, and using again \eqref{U:280}, we reduce to proving that $A\partial\phi$ and $A^2\phi$ belong to $L_t^{\frac{2}{1+\gamma}} L_x^{\frac{2}{2-s(m+1)}}(S_T)$. 

By Lemma \ref{U:Lemma} and \eqref{U:290}, $A_\pm^{\text{inh.}} \in L_t^{\frac{2}{s(m+1)}} L_x^{\frac{2}{1-s(m+1)}}(S_T)$, so
$$
  \norm{A^{\text{inh.}}\partial\phi}_{L_t^{\frac{2}{1+\gamma}} L_x^{\frac{2}{2-s(m+1)}}(S_T)}
  \lesssim
  \norm{A^{\text{inh.}}}_{L_t^{\frac{2}{s(m+1)}} L_x^{\frac{2}{1-s(m+1)}}(S_T)}
  \norm{\partial\phi}_{L_t^\infty L_x^2(S_T)}
  < \infty.
$$
Similarly, by \eqref{U:150}, $A_\pm^{(0)} \in L_t^{\frac{2r}{r-2}} L_x^{3r}(S_T)$ with $r = \frac{2}{3(1-s(m+1))}$, so we have
$$
  \bignorm{A^{(0)}\partial\phi}_{L_t^{\frac{2}{1+\gamma}} L_x^{\frac{2}{2-s(m+1)}}(S_T)}
  \lesssim
  \bignorm{A^{(0)}}_{L_t^{\frac{2r}{r-2}} L_x^{3r}(S_T)}
  \norm{\partial\phi}_{L_t^\infty L_x^2(S_T)}
  < \infty.
$$

Now consider $A^2\phi$. First, replacing $A$ by $A^{(0)}$, we can write
\begin{equation}\label{U:296}
  \bignorm{(A^{(0)})^2\phi}_{L_t^{\frac{2}{1+\gamma}} L_x^{\frac{2}{2-s(m+1)}}(S_T)}
  \lesssim
  \bignorm{A^{(0)}}_{L_t^\infty L_x^6(S_T)}^2
  \norm{\phi}_{L_t^\infty L_x^p(S_T)}
  < \infty,
\end{equation}
where $\frac{1}{p} = \frac16 + \frac{1-s(m+1)}{2}$, hence $2 \le p \le 6$. Second, recalling that $A_\pm^{\text{inh.}} \in L_t^{\frac{2}{s(m+1)}} L_x^{\frac{2}{1-s(m+1)}}(S_T)$, and noting that by the embeddings $H^{\frac12} \hookrightarrow L_x^3$ and $X_\pm^{0,\frac12 + \varepsilon} \hookrightarrow L_t^\infty L_x^2$ we also have $A_\pm^{\text{inh.}} \in L_t^\infty L_x^3(S_T)$, we can estimate
\begin{multline}\label{U:298}
  \norm{(A^{\text{inh.}})^2\phi}_{L_t^{\frac{2}{1+\gamma}} L_x^{\frac{2}{2-s(m+1)}}(S_T)}
  \\
  \lesssim
  \bignorm{A^{\text{inh.}}}_{L_t^{\frac{2}{s(m+1)}} L_x^{\frac{2}{1-s(m+1)}}(S_T)}
  \bignorm{A^{\text{inh.}}}_{L_t^\infty L_x^3(S_T)}
  \norm{\phi}_{L_t^\infty L_x^6(S_T)}
  < \infty.
\end{multline}
and this completes the proof of \eqref{U:220}.

\subsection{Proof of the difference estimate \eqref{U:144}}\label{U:300}

In this subsection, when we say that an estimate holds, we mean that it holds for $0 < \varepsilon \ll \delta \ll 1$. Subtracting the equations \eqref{L:100} for $(\phi,A)$ and $(\widetilde \phi, \widetilde A)$, and applying the linear estimate \eqref{L:140}, or rather the modification of it discussed in the paragraph following it, we see that
$$
  \Delta(T) 
  \le C_\varepsilon T^\varepsilon \bigl(\alpha(T)+\beta(T)\bigr),
$$
where
\begin{align*}
  \alpha(T) &= \norm{\mathfrak M(\phi_+,\phi_-,A_+,A_-) - \mathfrak M(\widetilde\phi_+,\widetilde\phi_-,\widetilde A_+,\widetilde A_-)}_{X_\pm^{-\varepsilon,-\frac12+2\varepsilon}(S_T)},
  \\
  \beta(T) &= \norm{\abs{\nabla}^{-1} \left( \mathfrak N(\phi_+,\phi_-,A_+,A_-) - \mathfrak N(\widetilde\phi_+,\widetilde\phi_-,\widetilde A_+,\widetilde A_-) \right)}_{X_\pm^{1-\delta,0}(S_T)}.
\end{align*}
Thus, it suffices to show that
$$
  \alpha(T), \beta(T) \le K \Delta(T),
$$
where $K$ is a polynomial expression in the norms appearing in the left hand sides of~\eqref{U:110} and~\eqref{U:120}, as well as the corresponding norms for the solution $(\widetilde \phi, \widetilde A)$. But since $\mathfrak M$ and $\mathfrak N$, defined as in \eqref{L:110}, are multilinear operators, it suffices to prove
\begin{align}
  \label{U:340}
  \norm{\mathfrak M(\phi_+,\phi_-,A_+,A_-)}_{X_\pm^{-\varepsilon,-\frac12+2\varepsilon}}
  &\lesssim \norm{A}\norm{\phi} + \norm{A}^2\norm{\phi},
  \\
  \label{U:350}
  \norm{\abs{\nabla}^{-1}
  \mathfrak N(\phi_+,\phi_-,A_+,A_-)}_{X_\pm^{1-\delta,0}}
  &\lesssim \norm{\phi}^2 + \norm{A}\norm{\phi}^2,
\end{align}
for all $\phi_+,\phi_-,A_+,A_- \in \mathcal S(\R^{1+3})$ (the analogous estimates restricted to $S_T$ then follow immediately), where we write
\begin{align*}
  \norm{\phi}
  &=
  \sum_\pm \norm{\phi_\pm}_{X_\pm^{1-\varepsilon,\frac12+\varepsilon}},
  \\
  \norm{A}
  &=
  \sum_\pm \min\left(\norm{\abs{\nabla}A_\pm}_{X_\pm^{0,1-\delta}},\norm{A_\pm}_{X_\pm^{1-\delta,1-\delta}}\right).
\end{align*}

To prove \eqref{U:340} and \eqref{U:350}, we proceed as in the proof of \eqref{L:200} and \eqref{L:210} (there is some headroom in the proof of the latter two). In the following subsections we consider one by one the bilinear and trilinear terms in $\mathfrak M$ and $\mathfrak N$.

\subsubsection{Bilinear terms in $\mathfrak M$} These have a null structure. Proceeding as in subsection \ref{L:224}, we then need to check the estimates
\begin{align}
  \label{U:400}
  I
  &\lesssim \norm{\abs{\nabla} u}_{X_{\pm_1}^{0,1-\delta}}
  \norm{v}_{X_{\pm_2}^{1-\varepsilon,\frac12+\varepsilon}},
  \\
  \label{U:410}
  I
  &\lesssim
  \norm{u}_{X_{\pm_1}^{1-\delta,1-\delta}}
  \norm{v}_{X_{\pm_2}^{1-\varepsilon,\frac12+\varepsilon}},
\end{align}
where $\widehat u, \widehat v \ge 0$ and
$$
  I
  =
  \norm{
  \iint
  \frac{\sigma(\eta,\xi-\eta)}{\angles{\xi}^\varepsilon\angles{\abs{\tau}-\abs{\xi}}^{\frac12-2\varepsilon}}
  \widehat u(\lambda,\eta)
  \widehat v(\tau-\lambda,\xi-\eta) \d\lambda \d\eta
  }_{L^2_{\tau,\xi}},
$$
with a symbol $\sigma$ satisfying the estimate in Lemma \ref{N:Lemma1}. Using Lemma \ref{L:Lemma} with $s=\frac12-\varepsilon-\delta$, we then get $I \lesssim I_1 + I_2 + I_3 + I_4$, where
\begin{align*}
  I_1
  &=
  \norm{
  \iint
  \angles{\xi}^{-\varepsilon} \widehat u(\lambda,\eta)
  \widehat v(\tau-\lambda,\xi-\eta)
  \d\lambda \d\eta
  }_{L^2_{\tau,\xi}},
  \\
  I_2
  &=
  \norm{
  \iint
  \frac{\widehat u(\lambda,\eta)
  \abs{\xi-\eta}\widehat v(\tau-\lambda,\xi-\eta) }
  {\angles{\xi}^\varepsilon\min(\angles{\eta},\angles{\xi-\eta})^{\frac12-2\varepsilon}}
  \d\lambda \d\eta
  }_{L^2_{\tau,\xi}},
  \\
  I_3
  &=
  \norm{
  \iint
  \frac{\angles{-\lambda\pm_1\abs{\eta}}^{\frac12}\widehat u(\lambda,\eta)
  \abs{\xi-\eta}\widehat v(\tau-\lambda,\xi-\eta) }
  {\angles{\xi}^\varepsilon\angles{\abs{\tau}-\abs{\xi}}^{\frac12-2\varepsilon}
  \min(\angles{\eta},\angles{\xi-\eta})^{\frac12}}
  \d\lambda \d\eta
  }_{L^2_{\tau,\xi}}
  \\
  I_4
  &=
  \norm{
  \iint
  \frac{\widehat u(\lambda,\eta)
  \angles{-(\tau-\lambda)\pm_2\abs{\xi-\eta}}^{\frac12}
  \abs{\xi-\eta}\widehat v(\tau-\lambda,\xi-\eta) }
  {\angles{\xi}^\varepsilon\angles{\abs{\tau}-\abs{\xi}}^{\frac12-2\varepsilon}
  \min(\angles{\eta},\angles{\xi-\eta})^{\frac12}}
  \d\lambda \d\eta
  }_{L^2_{\tau,\xi}}
\end{align*}
Using also \eqref{L:220}, we can thus reduce \eqref{U:410} to the estimates
\begin{equation}\label{U:420}
\left\{
\begin{aligned}
  \norm{uv}_{H^{-\varepsilon,0}}
  &\lesssim
  \norm{u}_{H^{1-\delta,1-\delta}}
  \norm{v}_{H^{1-\varepsilon,\frac12+\varepsilon}}
  \\
  \norm{uv}_{H^{-\varepsilon,0}}
  &\lesssim
  \norm{u}_{H^{\frac32-2\varepsilon-\delta,1-\delta}}
  \norm{v}_{H^{-\varepsilon,\frac12+\varepsilon}}
  \\
  \norm{uv}_{H^{-\varepsilon,0}}
  &\lesssim
  \norm{u}_{H^{1-\delta,1-\delta}}
  \norm{v}_{H^{\frac12-3\varepsilon,\frac12+\varepsilon}}
  \\
  \norm{uv}_{H^{-\varepsilon,0}}
  &\lesssim
  \norm{u}_{H^{\frac32-\delta,\frac12-\delta}}
  \norm{v}_{H^{-\varepsilon,\frac12+\varepsilon}}
  \\
  \norm{uv}_{H^{-\varepsilon,0}}
  &\lesssim
  \norm{u}_{H^{1-\delta,\frac12-\delta}}
  \norm{v}_{H^{\frac12-\varepsilon,\frac12+\varepsilon}}
  \\
  \norm{uv}_{H^{-\varepsilon,-\frac12+2\varepsilon}}
  &\lesssim
  \norm{u}_{H^{\frac32-\delta,1-\delta}}
  \norm{v}_{H^{-\varepsilon,0}}
  \\
  \norm{uv}_{H^{-\varepsilon,-\frac12+2\varepsilon}}
  &\lesssim
  \norm{u}_{H^{1-\delta,1-\delta}}
  \norm{v}_{H^{\frac12-\varepsilon,0}}.
\end{aligned}
\right.
\end{equation}
All these hold by Theorem \ref{L:Thm2}. This proves~\eqref{U:410}, and also~\eqref{U:400} in the case where $u$ has Fourier support in $\abs{\xi} > 1$; if, on the other hand, $\widehat u$ is supported in $\abs{\xi} \le 1$, then denoting the spatial frequency of $\widehat v$ by $\eta$, so that $\widehat{uv}$ has spatial frequency $\xi+\eta$, we have $\angles{\xi+\eta} \sim \angles{\eta}$, hence $$I \lesssim \norm{u\angles{\nabla}v}_{H^{-\varepsilon,0}} \sim \norm{u \angles{\nabla}^{1-\varepsilon} v}_{L^2} \le \norm{u}_{L^\infty} \norm{\angles{\nabla}^{1-\varepsilon} v}_{L^2},$$ so \eqref{L:318} yields~\eqref{U:400}.

\subsubsection{Trilinear term in $\mathfrak M$} To handle the term $A_\mu A^\mu \phi$, we need the estimates
\begin{align}
  \label{U:500}
  \norm{u^2 v}_{H^{-\varepsilon,0}}
  &\lesssim
  \norm{\abs{\nabla}u}_{H^{0,1-\delta}}^2
  \norm{v}_{H^{1-\varepsilon,\frac12+\varepsilon}}
  \\
  \label{U:510}
  \norm{u^2 v}_{H^{-\varepsilon,0}}
  &\lesssim
  \norm{u}_{H^{1-\delta,1-\delta}}^2
  \norm{v}_{H^{1-\varepsilon,\frac12+\varepsilon}}.
\end{align}
For \eqref{U:500} we use H\"older's inequality and the potential inequality $\norm{f}_{L^2_x} \lesssim \norm{\abs{\nabla}^\varepsilon f}_{L_x^p}$ where $\frac{1}{p} = \frac12 + \frac{\varepsilon}{3}$, obtaining
$$
  \norm{u^2 v}_{H^{-\varepsilon,0}}
  \lesssim \norm{u^2 v}_{L_t^2 L_x^p}
  \le
  \norm{u}_{L_t^\infty L_x^6}^2 \norm{v}_{L_t^2 L_x^r},
$$
where $\frac{1}{r} = \frac16 + \frac{\varepsilon}{3}$. Applying the potential inequality $\norm{f}_{L^r_x} \lesssim \norm{\abs{\nabla}^{1-\varepsilon} f}_{L_x^2}$ as well as the embeddings $\dot H^1 \hookrightarrow L^6$ and $H^{0,\frac12+\varepsilon} \hookrightarrow L_t^\infty L_x^2$, we get \eqref{U:500}. To prove \eqref{U:510}, we use Theorem \ref{L:Thm2} twice, obtaining
$$
  \norm{u^2 v}_{H^{-\varepsilon,0}}
  \lesssim
  \norm{u^2}_{H^{\frac12+\varepsilon,0}}
  \norm{v}_{H^{1-\varepsilon,\frac12+\varepsilon}}
  \lesssim
  \norm{u}_{H^{1-\delta,1-\delta}}^2
  \norm{v}_{H^{1-\varepsilon,\frac12+\varepsilon}}.
$$

\subsubsection{Bilinear terms in $\mathfrak N$}

For these we need
\begin{equation}\label{U:550}
  \norm{\abs{\nabla}^{-1}(uv)}_{H^{1-\delta,0}}
  \lesssim
  \norm{u}_{H^{1-\varepsilon,\frac12+\varepsilon}}
  \norm{v}_{H^{-\varepsilon,\frac12+\varepsilon}}.
\end{equation}

If $uv$ has spatial Fourier support in $\abs{\xi} \le 1$, then the left hand side is comparable to $\norm{\abs{\nabla}^{-1}(uv)}_{L^2} \sim\norm{\abs{\nabla}^{-1}(\angles{\nabla}^\varepsilon u \cdot \angles{\nabla}^{-\varepsilon} v)}_{L^2}$, which we dominate by
$$
  \norm{\angles{\nabla}^\varepsilon u \cdot \angles{\nabla}^{-\varepsilon} v}_{L_t^2 L_x^{\frac65}}
  \le
  \norm{\angles{\nabla}^\varepsilon u}_{L_t^2 L_x^3}
  \norm{\angles{\nabla}^{-\varepsilon} v}_{L_t^\infty L_x^{2}}
  \lesssim
  \norm{u}_{H^{\frac12+\varepsilon,0}}
  \norm{v}_{H^{-\varepsilon,\frac12+\varepsilon}}.
$$

If $\abs{\xi} > 1$ in the spatial Fourier support of $uv$, we can replace the left hand side of \eqref{U:550} by $\norm{uv}_{H^{-\delta,0}}$, and the desired estimate holds by Theorem \ref{L:Thm2}. It is at this point that we need the assumption $\varepsilon \ll \delta$.

\subsubsection{Trilinear terms in $\mathfrak N$}

These are schematically of the form $A\phi^2$. Splitting into the output regions $\abs{\xi} \le 1$ and $\abs{\xi} > 1$, we see that it suffices to prove
\begin{align}
  \label{U:630}
  \norm{u^2 v}_{L_t^2 L_x^{\frac65}}
  &\lesssim
  \norm{u}_{H^{1-\varepsilon,\frac12+\varepsilon}}^2
  \norm{\abs{\nabla} v}_{H^{0,1-\delta}},
  \\
  \label{U:640}
  \norm{u^2 v}_{L_t^2 L_x^{\frac65}}
  &\lesssim
  \norm{u}_{H^{1-\varepsilon,\frac12+\varepsilon}}^2
  \norm{v}_{H^{1-\delta,1-\delta}},
  \\
  \label{U:650}
  \norm{u^2 v}_{L^2}
  &\lesssim
  \norm{u}_{H^{1-\varepsilon,\frac12+\varepsilon}}^2
  \norm{\abs{\nabla} v}_{H^{0,1-\delta}},
  \\
  \label{U:660}
  \norm{u^2 v}_{H^{-\delta,0}}
  &\lesssim
  \norm{u}_{H^{1-\varepsilon,\frac12+\varepsilon}}^2
  \norm{v}_{H^{1-\delta,1-\delta}}.
\end{align}
The first two follow from H\"older's inequality and Sobolev embedding (much as in the proof of \eqref{L:330}). Next, write
$$
  \norm{u^2 v}_{L^2}
  \le
  \norm{u^2}_{L_t^2 L_x^3}
  \norm{v}_{L_t^\infty L_x^6}
  \lesssim
  \norm{u^2}_{H^{\frac12,0}}
  \norm{\abs{\nabla} v}_{H^{0,1-\delta}}.
$$
Since $\norm{u^2}_{H^{\frac12,0}} \lesssim \norm{u}_{H^{1-\varepsilon,\frac12+\varepsilon}}^2$ by Theorem \ref{L:Thm2}, this proves \eqref{U:650}. Finally, to prove~\eqref{U:660} we apply Theorem \ref{L:Thm2} twice, obtaining
$$
  \norm{u^2 v}_{H^{-\delta,0}}
  \lesssim
  \norm{u^2}_{H^{\frac12+\varepsilon,0}}
  \norm{v}_{H^{1-\delta,1-\delta}}
  \lesssim
  \norm{u}_{H^{1-\varepsilon,\frac12+\varepsilon}}^2
  \norm{v}_{H^{1-\delta,1-\delta}}.
$$

This concludes the proof of uniqueness.

\bibliographystyle{amsplain} 
\bibliography{MKGbibliography}

\end{document}